\documentclass{amsart}
\usepackage{amsmath}
\usepackage{amsthm}
\usepackage{amssymb}

\usepackage{diagrams}
\input xy
\xyoption{all}

 \newtheorem{thm}{Theorem}[section]
 
 \newtheorem{lem}[thm]{Lemma}
 \newtheorem{prop}[thm]{Proposition}
 \theoremstyle{definition}
 \newtheorem{defn}[thm]{Definition}
 \theoremstyle{remark}
 \newtheorem{rem}[thm]{Remark}
 \numberwithin{equation}{section}
\newtheorem{ex}[thm]{Example}
 \numberwithin{equation}{section}
 
\newcommand{\pn}{\noindent}

\newcommand{\ZZ}{\mathbb{Z}}
\newcommand{\QQ}{\mathbb{Q}}

\newcommand{\LL}{\mathbb{L}}

\newcommand{\Hom}{\mathrm{Hom}}
\newcommand{\HOM}{\mathrm{HOM}}
\newcommand{\Aut}{\mathrm{Aut}}
\newcommand{\Ext}{\mathrm{Ext}}
\newcommand{\Biext}{\mathrm{Biext}}
\newcommand{\cExt}{\mathcal{E}xt}

\newcommand{\uAut}{\underline{\mathrm{Aut}}}
\newcommand{\bBiext}{\mathrm{\mathbf{Biext}}}
\newcommand{\bPicard}{\mathrm{Picard}}
\newcommand{\bExt}{\mathrm{\mathbf{Ext}}}
\newcommand{\bTors}{\mathrm{\mathbf{Tors}}}

\newcommand{\coker}{\mathrm{coker}}
\newcommand{\im}{\mathrm{im}}

\newcommand{\T}{\mathrm{T}}
\newcommand{\h}{\mathrm{H}}
\newcommand{\R}{\mathrm{R}}
\newcommand{\E}{\mathrm{E}}
\newcommand{\rL}{\mathrm{L}}

\newcommand{\Tot}{\mathrm{Tot}}
\newcommand{\spec}{{\mathrm{Spec}}\,}

\newcommand{\cD}{\mathcal{D}}
\newcommand{\cC}{\mathcal{C}}
\newcommand{\cH}{\mathcal{H}}
\newcommand{\cK}{\mathcal{K}}

\newcommand{\bS}{\textbf{S}}

\newcommand{\eic}{\mathcal{E}}

\newcommand{\pic}{\mathcal{P}}
\newcommand{\qic}{\mathcal{Q}}
\newcommand{\ric}{\mathcal{R}}
\newcommand{\gic}{\mathcal{G}}
\newcommand{\bic}{\mathcal{B}}
\newcommand{\dic}{\mathcal{D}}
\newcommand{\lic}{\mathcal{L}}
\newcommand{\kic}{\mathcal{K}}

\newcommand{\cO}{\mathcal{O}}
\newcommand{\DMeffgm}{\mathrm{DM}^{\mathrm{eff}}_{\mathrm{gm}}}
\begin{document}

\title[biextensions and 1-motives]
{Homological interpretation of extensions and biextensions of 1-motives}

\author{Cristiana Bertolin}

\address{Dip. di Matematica, Universit\`a di Torino, Via Carlo Alberto 10, 
I-10123 Torino}
\email{cristiana.bertolin@googlemail.com}



%\date{}
%\dedicatory{}

%\commby{Cristiana Bertolin}

%%% ----------------------------------------------------------------------

\begin{abstract}
Let $k$ be a separably closed field. Let $K_i=[A_i \stackrel{u_i}{\rightarrow} B_i]$
(for $i=1,2,3$) be three 1-motives defined over $k$. 
We define the geometrical notions
of extension of $K_1$ by $K_3$ and of biextension of $(K_1,K_2)$ by $K_3$.
We then compute the homological interpretation of these new geometrical notions: namely, the group ${\Biext}^0(K_1,K_2;K_3)$ of automorphisms of any biextension
of $(K_1,K_2)$ by $K_3$ is canonically isomorphic to the group ${\Ext}^0(K_1 {\buildrel {\scriptscriptstyle \LL} \over \otimes} K_2,K_3)$, and the group ${\Biext}^1(K_1,K_2;K_3)$ of isomorphism
 classes of biextensions of $(K_1,K_2)$ by $K_3$ is canonically isomorphic to the group ${\Ext}^1(K_1 {\buildrel {\scriptscriptstyle \LL} \over \otimes} K_2,K_3).$
\end{abstract}

%%% ----------------------------------------------------------------------

\maketitle

%%% ----------------------------------------------------------------------

\tableofcontents

\section*{Introduction}

Let $k$ be a separably closed field and let $S=\spec (k)$. A 1-motive $K=[u:A \rightarrow B]$ over $S$ consists of an $S$-group scheme $A$ which is locally for the \'etale
topology a constant group scheme defined by a finitely generated free
$\ZZ \,$-module, an extension $B$ of an abelian $S$-scheme by an $S$-torus, and a morphism $u:A \rightarrow B$ of $S$-group schemes. Since the field $k$ is separably closed, remark that $A={\ZZ}^r$ with $r \geq 0.$

Let $\bS$ be the big fppf site over $S$.
A 1-motive $K=[u:A \rightarrow B]$ can be viewed also as a complex of abelian sheaves on $\bS$ concentrated in two consecutive degrees. A morphism of 1-motives is a morphism of complexes of commutative $S$-group schemes (see \cite{R}, in particular Lemma 2.3.2)

Let $K_i= [u_i: A_i \rightarrow B_i]$ (for $i=1,2,3$) be three 1-motives defined over $S$. In this paper we introduce the geometrical notions of extension of $K_1$ by $K_3$ and of biextension of $(K_1,K_2)$ by $K_3$.
We then compute the homological interpretation of these new geometrical notions. More precisely, if ${\Biext}^0(K_1,K_2;K_3)$ is the group of automorphisms of any biextension of $(K_1,K_2)$ by $K_3$,
 ${\Biext}^1(K_1,K_2;K_3)$ is the group of isomorphism classes of biextensions of $(K_1,K_2)$ by $K_3$, 
 ${\cExt}^0(K_1,K_3)$ is the group of automorphisms of any extension of $K_1$ by $K_3$, and
 ${\cExt}^1(K_1,K_3)$ is the group of isomorphism classes of extensions of $K_1$ by $K_3$,
 then we prove

\begin{thm}\label{mainthm}
 We have the following canonical isomorphisms

(a) ${\Biext}^1(K_1,K_2;K_3)  \cong   {\Ext}^1(K_1 {\buildrel {\scriptscriptstyle \LL} \over \otimes} K_2,K_3)= 
{\Hom}_{\cD(\bS)}\big(K_1{\buildrel {\scriptscriptstyle \LL} \over \otimes}K_2,K_3[1]\big),$

(b) ${\Biext}^0(K_1,K_2;K_3)     \cong   {\Ext}^0(K_1 {\buildrel {\scriptscriptstyle \LL} \over \otimes} K_2,K_3) = 
{\Hom}_{\cD(\bS)}\big(K_1{\buildrel {\scriptscriptstyle \LL} \over \otimes}K_2,K_3\big), $

(c) ${\cExt}^1(K_1,K_3)  \cong   {\Ext}^1(K_1 ,K_3)= {\Hom}_{\cD(\bS)}(K_1,K_3 [1]),$

(d) ${\cExt}^0(K_1,K_3)     \cong   {\Ext}^0(K_1 , K_3) = {\Hom}_{\cD(\bS)}(K_1, K_3), $\\
where $K_1 {\buildrel{\scriptscriptstyle \LL} \over \otimes}K_2$ is the derived functor of the functor $K_2 \rightarrow K_1 \otimes K_2$ in the derived category
 $\cD(\bS)$ of complexes of abelian sheaves on $\bS$.
\end{thm}

The homological interpretation (c)-(d) of extensions of 1-motives is a special case of the homological interpretation (a)-(b) of biextensions of 1-motives:
in fact, if $K_2=[0 \rightarrow {\ZZ}]$

(1) the category of biextensions of $(K_1, [0 \rightarrow {\ZZ}])$ by $K_3$
is equivalent to the category  of extensions of $K_1$ by $K_3$, and

(2) in the derived category 
$ {\Ext}^i(K_1 {\buildrel {\scriptscriptstyle \LL} \over \otimes} [0 \rightarrow {\ZZ}],K_3) \cong {\Ext}^i(K_1,K_3) $ for $i=0,1$.

Applications of Theorem~\ref{mainthm} are given by the isomorphism
\begin{equation}\label{iso}
    {\Biext}^1(K_1,K_2;K_3) \cong {\Ext}^1(K_1 {\buildrel {\scriptscriptstyle \LL} \over \otimes} K_2,K_3) = {\Hom}_{\cD(\cC)}(K_1 {\buildrel {\scriptscriptstyle \LL} \over \otimes} K_2, K_3[1])
\end{equation}
which makes explicit the link between biextensions and bilinear morphisms. A classical example of this link is given by the Poincar\'e biextension of an abelian variety which defines the Weil pairing on the Tate modules. Other examples are furnished by \cite{Be08} and \cite{BM}, where we prove that
\begin{itemize}
  \item the group of isomorphism classes of biextensions of $(K_1,K_2)$ by $K_3$ is isomorphic to
 the group of morphisms of the category $\mathcal {MHS}$ of mixed Hodge structures from the tensor product ${\T}_{\h}(K_1) \otimes {\T}_{\h}(K_2)$ of the Hodge realizations of $K_1$ and $K_2$ to the Hodge realization ${\T}_{\h}(K_3)$ of $K_3$:
\[
    {\Biext}^1(K_1,K_2;K_3) \cong {\Hom}_{\mathcal{MHS}}({\T}_{\h}(K_1) \otimes {\T}_{\h}(K_2) , {\T}_{\h}(K_3) ).
\]
  \item modulo isogenies the group of isomorphism classes of biextensions of $(K_1,K_2)$ by $K_3$ is isomorphic to
 the group of morphisms of the category ${\mathcal{MR}}_{\ZZ}(k)$ of mixed realizations with integral structure from the tensor product ${\T}(K_1)\otimes {\T}(K_2)$ of the realizations of $K_1$ and $K_2$ to the realization ${\T}(K_3)$ of $K_3$:
\[
    {\Biext}^1(K_1,K_2;K_3) \otimes {\QQ} \cong {\Hom}_{{\mathcal{MR}}_{\ZZ}(k)}\big(
{\T}(K_1)\otimes {\T}(K_2), {\T}(K_3)\big).
\]
Following Deligne's philosophy of motives described in ~\cite{D3} 1.11,
this isomorphism means that the notion of biextensions of 1-motives furnishes the geometrical origin of the morphisms of ${\mathcal{MR}}_{\ZZ}(k)$
from the tensor product of the realizations of two 1-motives to the realization of another 1-motive, which are therefore motivic morphisms.
  \item modulo isogenies the group of isomorphism classes of biextensions of $(K_1,K_2)$ by $K_3$ is isomorphic to
 the group of morphisms of Voevodsky's triangulated category ${\DMeffgm}$ of effective geometrical motives with rational coefficients from the tensor product ${\cO}(K_1) \otimes {\cO}(K_2)$ of the images of $K_1$ and $K_2$ in the category ${\DMeffgm}$ to the image ${\cO}(K_3)$ of $K_3$ in ${\DMeffgm}$:
\[ 
{\Biext}^1(K_1,K_2;K_3) \otimes {\QQ} \cong {\Hom}_{{\DMeffgm} (k,{\QQ})}({\cO}(K_1) \otimes {\cO}(K_2), {\cO}(K_3) ). 
\]
In \cite{BM} we have used Theorem~\ref{mainthm} (a) in order to show the above isomorphism.
 \end{itemize}

In \cite{Be10} and \cite{Be11} we have introduced the notions of extension and biextension for arbitrary length 2 complexes of abelian sheaves and we have computed their homological interpretation. The definitions and the results of \cite{Be10} and \cite{Be11} are a generalization of the definitions and the results of this paper (in particular of Theorem~\ref{mainthm}) to arbitrary length 2 complexes of abelian sheaves.

The idea of the proof of Theorem~\ref{mainthm} is the following one: Let $K=[A \stackrel{u}{\rightarrow} B]$ be a 1-motives and let ${\rL}..$ be a complex of 1-motives $R \rightarrow Q \rightarrow P \rightarrow 0$. To the complex $K$ and to the bicomplex ${\rL}..$ we associate a category $\Psi_{{\rL}..}(K)$ which has the following \emph{homological description}:
\begin{equation}\label{homo-intro}
    \Psi_{{\rL}..}^i(K) \cong {\Ext}^i({\Tot}({\rL}..),K) \qquad (i=0,1)
\end{equation}
where $\Psi_{{\rL}..}^0(K)$ is the group of automorphisms of any object of $\Psi_{{\rL}..}(K)$ and $\Psi_{{\rL}..}^1(K)$ is the group of isomorphism
 classes of objects of $\Psi_{{\rL}..}(K)$. Then, to any 1-motive $K=[A \stackrel{u}{\rightarrow} B]$ we associate a canonical flat partial resolution ${\rL}..(K)$ whose components are direct sums of objects of the kind ${\ZZ}[A]$ and ${\ZZ}[B]$. Here ``partial resolution'' means that we have an isomorphism between the homology groups of $K$ and of this partial resolution only in degree 1 and 0. This is enough
for our goal since only the groups $\Ext^1$ and $\Ext^0$ are involved in the statement of Theorem~\ref{mainthm}. Consider now three 1-motives $K_i$ (for $i=1,2,3$). The categories $\Psi_{{\rL}..(K_1)}(K_3)$ and $\Psi_{{\rL}..(K_1)\otimes {\rL}..(K_2)}(K_3)$ admit the following \emph{geometrical description}:
 \begin{eqnarray}\label{geo-intro}
   \Psi_{{\rL}..(K_1)}(K_3) & \simeq &  {\bExt}(K_1,K_3)  \\
 \nonumber \Psi_{{\rL}..(K_1)\otimes {\rL}..(K_2)}(K_3) & \simeq &  {\bBiext}(K_1,K_2;K_3)
\end{eqnarray}
Putting together this geometrical description~(\ref{geo-intro}) with the homological description~(\ref{homo-intro}), we get the proof of Theorem~\ref{mainthm}.

%------------------------------------------------------------------------

%\section*{Acknowledgment}
%I am very grateful to the anonymous referee for his comments. Moreover I thank also Luc Illusie for his help concerning Lemma \ref{illusie}.

\section*{Notation}

In this paper, $k$ is a separably closed field, $S=\spec (k)$ and $\bS$ is the big fppf site over $S$.
 If $I$ is a sheaf on $\bS$, we denote by ${\ZZ}[I]$ the free $\ZZ$-module generated by $I$ (see \cite{SGA4} Expos\'e IV 11).

Denote by $\cK(\bS)$ the category of complexes of abelian sheaves on the site $\bS$: all complexes that we consider in this paper are cochain complexes.
Let $\cK^{[-1,0]}(\bS)$ be the subcategory of $\cK(\bS)$ consisting of complexes $K=(K^i)_i$ such that $K^i=0$ for $i \not= -1$ or $0$. The good truncation $ \tau_{\leq n} K$ of a complex $K$ of $\cK(\bS)$ is the following complex: $ (\tau_{\leq n} K)^i= K^i$ for $i <n,  (\tau_{\leq n} K)^n= \ker(d^n)$
and $ (\tau_{\leq n} K)^i= 0$ for $i > n.$ For any $i \in {\ZZ}$, the shift functor $[i]:\cK(\bS) \rightarrow \cK(\bS) $ acts on a complex $K=(K^n)_n$ as $(K[i])^n=K^{i+n}$ and $d^n_{K[i]}=(-1)^{i} d^{n+i}_{K}$.

Denote by $\cD(\bS)$ the derived category of the category of abelian sheaves on $\bS$, and let $\cD^{[-1,0]}(\bS)$ be the subcategory of $\cD(\bS)$ consisting of complexes $K$ such that ${\h}^i (K)=0$ for $i \not= -1$ or $0$. If $K$ and $K'$ are complexes of $\cD(\bS)$, the group ${\Ext}^i(K,K')$ is by definition ${\Hom}_{\cD(\bS)}(K,K'[i])$ for any $i \in {\ZZ}$. Let ${\R}{\Hom}(-,-)$ be the derived functor of the bifunctor ${\Hom}(-,-)$. The cohomology groups\\ ${\h}^i\big({\R}{\Hom}(K,K') \big)$ of 
${\R}{\Hom}(K,K')$ are isomorphic to ${\Hom}_{\cD(\bS)}(K,K'[i])$.

\section{Extensions and biextensions of 1-motives}

Let $G$ be abelian sheaf on $\bS$. A \textbf{$G$-torsor} is a sheaf on $\bS$ endowed with an action of $G$, which is locally isomorphic to $G$ acting on itself by translation.

Let $P,G$ be abelian sheaves on $\bS$. An \textbf{extension of $P$ by $G$} is an exact sequence 
\[0 \longrightarrow G \longrightarrow E \longrightarrow P\longrightarrow0.\]
Since in this paper we consider only commutative extensions, $E$ is in fact an abelian sheaf on $\bS$.
We denote by ${\bExt}(P,G)$ the category of extensions
of $P$ by $G$. It is a classical result that the Baer sum of extensions defines a group law for the objects of the category ${\bExt}(P,G)$, which is therefore a strictly commutative Picard category.

Let $P,G$ be abelian sheaves on $\bS$. Denote by $m:P \times P \rightarrow P$ the group law of $P$ and by $pr_i:P \times P \rightarrow P$ with $i=1,2$ the two projections of $P \times P$  in $P$.
According to~\cite{SGA7} Expos\'e VII 1.1.6 and 1.2, the category of extensions of $P$ by $G$ is equivalent to the category of 4-tuples
$(P,G,E,\varphi)$, where $E$ is a $G_P$-torsor over $P$, and $\varphi: pr_1^*E \wedge pr_2^*E \rightarrow m^* E$ is an isomorphism of torsors over $P\times P$ satisfying some associativity and commutativity conditions (see~\cite{SGA7} Expos\'e VII diagrams (1.1.4.1) and (1.2.1)):
\begin{eqnarray}\label{1.1.6}
\nonumber  {\bExt}(P,G)  &\simeq& \Big\{ (P,G,E,\varphi)~~ \Big|  ~~ E=G_P\mathrm{-torsor~over~} P \mathrm{~and} \\
 &&  \varphi: pr_1^*E \wedge pr_2^*E \cong m^* E \mathrm{~with~ ass.~ and~ comm.~ conditions} \Big\}.
\end{eqnarray}
Here $pr_i^*E$ is the pull-back of $E$ via the projection $pr_i:P \times P \rightarrow P$ for $i=1,2$ and $pr_1^*E \wedge pr_2^*E$ is the contracted product of $pr_1^*E$ and $pr_2^*E$ (see 1.3 Chapter III \cite{G}). 
It will be useful in what follows to look at the isomorphism of torsors
$\varphi$ as an associative and commutative group law on the fibres:
\[ +: E_p ~ E_{p'} \longrightarrow E_{p+p'} \]
where $p,p'$ are points of $P(U)$ with $U$ an $S$-scheme.

Let $I$ be a sheaf on $\bS$ and let $G$ be an abelian sheaf on $\bS$. Concerning extensions of free commutative groups, by~\cite{SGA7} Expos\'e VII 1.4
  the category of extensions of ${\ZZ}[I]$ by $G$ is equivalent to the category of $G_I$-torsors over $I$:
\begin{equation}\label{1.4}
    {\bExt}({\ZZ}[I],G) \simeq {\bTors}(I,G_I).
\end{equation}

Let $P,Q$ and $G$ be abelian sheaves on $\bS$. A \textbf{biextension of $(P,Q)$ by $G$} is a $G_{P\times Q}$-torsor $B$ over $P\times Q$, endowed with a structure of commutative extension of $Q_P$
by $G_P$ and a structure of commutative extension of $P_Q$ by $G_Q,$ which are compatible one with another (for the definition of compatible extensions see~\cite{SGA7} Expos\'e VII D\'efinition 2.1). If $m_P, p_1, p_2$ (resp. $m_Q, q_1, q_2$) denote the three morphisms $P \times P \times Q
\rightarrow P \times Q$ (resp. $P \times Q \times Q
\rightarrow P \times Q$) deduced from the three morphisms $ P \times P
\rightarrow P$ (resp. $Q \times Q \rightarrow Q $) group law, first and second projection, the equivalence of categories~(\ref{1.1.6}) furnishes the following equivalent definition:
 a biextension of $(P,Q)$ by $G$ is a $G_{P\times Q}$-torsor $B$ over $P\times Q$ endowed with two isomorphisms of torsors
\[\varphi: p_1^*E ~ p_2^*E \longrightarrow m_P^* E \qquad \qquad \qquad
\psi:q_1^*E ~ q_2^*E \longrightarrow m_Q^* E\]
over $P \times P \times Q$ and $P \times Q \times Q$ respectively, satisfying some associativity, commutativity and compatible conditions (see~\cite{SGA7} Expos\'e VII diagrams (2.0.5),(2.0.6),(2.0.8),\\
(2.0.9), (2.1.1)). As for extensions, we will look at the isomorphisms of torsors $\varphi$ and $\psi$ as two associative and commutative group laws on the fibres which are compatible with one another:
\[ +_{1}: E_{p,q} ~ E_{p',q} \longrightarrow E_{p+p',q} \qquad \qquad  +_{2}: E_{p,q} ~ E_{p,q'} \longrightarrow E_{p,q+q'}\]
where $p,p'$ (resp. $q,q'$) are points of $P(U)$ (resp. of $Q(U)$) with $U$ any sheaf on $\bS$.

Let $K_i=[u_i:A_i \rightarrow B_i]$ (for $i=1,2$) be two 1-motives defined over $S$.

\begin{defn}\label{def:ext}
 An \textbf{extension $(E,\beta,\gamma)$ of $K_1$ by $K_2$} consists of
\begin{enumerate}
    \item an extension $E$ of $B_1$ by $B_2$;
    \item a trivialization $\beta$ of the extension $u_1^*E$
    of $A_1$ by $B_2$ obtained as pull-back of the extension $E$ via $u_1: A_1 \rightarrow B_1$;
    \item a trivial extension $T=(T, \gamma)$ of $A_1$ by $A_2$ (i.e. an extension $T$ of $A_1$ by $A_2$ endowed with a trivialization $ \gamma $) and an isomorphism of extensions $\Theta : u_{2 \, *} T \rightarrow u_1^*E$ between the push-down via $u_2: A_2 \rightarrow B_2$ of $T$ and $u_1^*E$. Through this isomorphism the trivialization $ u_2 \circ \gamma$ of $u_{2 \, *} T$ is compatible with the trivialization $\beta$ of $u_1^*E$.
\end{enumerate}
\end{defn}
Condition (3) can be rewritten as
\par \quad (3') an homomorphism $\gamma : A_1 \rightarrow A_2$ such that $u_2 \circ \gamma$ is compatible with $\beta.$
Note that to have a trivialization $\beta: A_1 \rightarrow u_1^* E $ of $u_1^* E$ is the same thing as to have a lifting $\widetilde{\beta}: A_1 \rightarrow E $ of $u_1: A_1 \rightarrow B_1$. In fact,
if we denote $p:E \rightarrow B_1$ the canonical surjection of the extension $E$, a morphism $\widetilde{\beta}: A_1 \rightarrow E $ such that $p \circ \widetilde{\beta} =u_1$ induces a splitting $\beta: A_1 \rightarrow u_1^*E$ that composes with $u_1^*E \rightarrow E \stackrel{p}{\rightarrow} B_1$ to $u_1: A_1 \rightarrow B_1$, and vice versa.

\begin{rem}\label{rem:ext-extcomplex}
We can summarize the above definition with the following diagram with exact rows:
\[
\xymatrix{
  0  \ar[r] & B_2 \ar@{=}[d] \ar[r] & E \ar[r]  & B_1 \ar[r] &0 \\
 0  \ar[r] & B_2  \ar[r] & u_1^*E \cong u_{2 *} T \ar[r] \ar[u] & A_1 \ar@/^/[l]^{\beta} \ar[u]_{u_1} \ar@{=}[d]\ar[r]&0 \\
 0  \ar[r] & A_2 \ar[u]^{u_2} \ar[r] & T  \ar[r] \ar[u] & \ar@/^/[l]^{\gamma} A_1 \ar[r] &0 
}
\]
In particular, we observe that the short sequence of complexes in $\cK(\bS)$ 
$$0 \longrightarrow K_2 \longrightarrow [T \rightarrow E] \longrightarrow K_1 \longrightarrow 0$$
is exact. On the other hand if $0 \rightarrow K_2 \rightarrow G \rightarrow K_1 \rightarrow 0$ is a short exact sequence of $\cK(\bS)$, then the complex $G$ is an extension of 1-motives of $K_1$ by $K_2$ as defined in Definition \ref{def:ext}, i.e. $G$ is a complex of the kind $[T \rightarrow E]$, with $T$ a trivial extension of $A_1$ by $A_2$ and $E$ an extension of $B_1$ by $B_2$. In fact, over a separably closed field the groups ${\Ext}^1(A_1,A_2)$ and  
${\Ext}^1(A_1,B_2)$ are trivial. 
\end{rem}

Let $K_i=[A_i \stackrel{u_i}{\rightarrow}B_i]$ and
$K'_i=[A'_i \stackrel{u'_i}{\rightarrow}B'_i]$ (for $i=1,2$) be 1-motives defined over $S$. Let $(E,\beta,\gamma)$ be an extension of $K_1$ by $K_2$ and let $(E',\beta',\gamma')$ be an extension of $K'_1$ by $K'_2$.

\begin{defn}
 A \textbf{morphism of extensions}
\[(\underline{F},\underline{\Upsilon},\underline{\Phi} ): (E,\beta,\gamma)
\longrightarrow (E',\beta',\gamma')\]
\pn consists of
\begin{enumerate}
    \item a morphism $\underline{F}=(F,f_1,f_2):E \rightarrow E'$ from the extension $E$ to the extension $E'$. In particular, $F: E \rightarrow E'$
        is a morphism of the sheaves underlying $E$ and $E'$, and
        \[ f_1:B_1 \longrightarrow B'_1 \qquad f_2:B_2 \longrightarrow B'_2 \]
        \pn are morphisms of abelian sheaves on $\bS$;
    \item a morphism of extensions $\underline{\Upsilon} = (\Upsilon,g_1,f_2):u_1^*E \rightarrow {u'_1}^* E'$
        compatible with the morphism $\underline{F}=(F,f_1,f_2)$ and with the trivializations $\beta$ and $\beta'$. In particular, $\Upsilon:u_1^*E \rightarrow {u'_1}^* E'$ is a morphism of the sheaves underlying $u_1^*E$ and ${u'_1}^*E'$, and
        \[g_1: A_1 \longrightarrow A'_1 \]
         \pn is an morphism of abelian sheaves on $\bS$;
         \item a morphism of extensions $\underline{\Phi} = (\Phi,g_1,g_2):T \rightarrow T'$
        compatible with the morphism $\underline{\Upsilon} = (\Upsilon,g_1,f_2)$ and with the trivializations $\gamma$ and $\gamma'$. In particular, $\Phi:T \rightarrow T'$ is a morphism of the sheaves underlying $T$ and $T'$, and
        \[g_2: A_2 \longrightarrow A'_2 \]
         \pn is an morphism of abelian sheaves on $\bS$.
\end{enumerate}
\end{defn}

Condition (3) can be rewritten as 
\par \quad (3') an morphism $g_2:A_2 \rightarrow A'_2$ of abelian sheaves on $\bS$ compatible with $u_2$ and $u'_2$ (i.e. $u'_2 \circ g_2 =f_2 \circ u_2$) and such that
    \[ \gamma' \circ g_1 = g_2 \circ \gamma.\]
Explicitly, the compatibility of $\underline{\Upsilon}$ with $\underline{F}$, $\beta$ and $\beta'$ means that the following diagram is commutative:
\[\begin{array}{ccccc}
      A_1 & \stackrel{\beta}{\longrightarrow}  &   u_1^*E & \longrightarrow & E \\
  \scriptstyle{g_1}    \downarrow & & \scriptstyle{\Upsilon}    \downarrow & & \downarrow \scriptstyle{F}\\
       A'_1  & \stackrel{\beta'}{\longrightarrow}  &  {u'_1}^* E'  &  \longrightarrow  & E'. \\
      \end{array}
\]
The compatibility of $\underline{\Phi}$ with $\underline{\Upsilon}$, $\gamma$ and $\gamma'$ means that the following diagram is commutative:
\[\begin{array}{ccccccc}
      A_1 & \stackrel{\gamma}{\longrightarrow}  &  T & \longrightarrow &u_{2\,*} T & \stackrel{\Theta} \cong &  u_1^*E \\
  \scriptstyle{g_1}    \downarrow & & \scriptstyle{\Phi}    \downarrow & && & \downarrow \scriptstyle{\Upsilon}\\
       A'_1  & \stackrel{\gamma'}{\longrightarrow}  &  T' & \longrightarrow &u'_{2\,*} T' & \stackrel{\Theta'} \cong & {u'_1}^* E'. \\
      \end{array}
\]

We denote by ${\bExt}(K_1,K_2)$ the category of extensions
of $K_1$ by $K_2$. As for extensions of abelian sheaves, it is possible to define the Baer sum of extensions of 1-motives. This notion of sum furnishes a group law for the objects of the category ${\bExt}(K_1,K_2)$ which is therefore a strictly 
commutative Picard category (see~\cite{SGA7} Expos\'e VII 2.5).
The zero object $(E_0,\beta_0,\gamma_0)$ of ${\bExt}(K_1,K_2)$ with respect to this group law consists of
\begin{itemize}
  \item the trivial extension $E_0=B_1 \times B_2$ of $B_1$ by $B_2$, i.e. the zero object of ${\bExt}(B_1,B_2)$, and
  \item the trivialization $\beta_0=(id_{A_1},0)$ of the extension $u_1^*E_0=A_1 \times B_2$ of $A_1$ by $B_2$. We can consider $\beta_0$ as the lifting $(u_1,0): A_1 \rightarrow B_1 \times B_2$ of $u_1:A_1 \rightarrow B_1$.
  \item the trivial extension $T_0 $ of $A_1$ by $A_2$ (i.e. $T_0=(T_0, \gamma_0)$ with  $T_0= A_1 \times A_2$ and $\gamma_0=(id_{A_1},0)$) and the isomorphism of extension $\Theta_0 =(id_{A_1},id_{B_2}): u_{2 \, *} T_0 \rightarrow u_1^*E_0$.
\end{itemize}
Denote by ${\cExt}^0(K_1,K_2)$ the group of automorphisms of any object $(E,\beta, \gamma)$ of ${\bExt}(K_1,K_2)$. It is canonically isomorphic 
to the group of automorphisms ${\Aut}(E_0,\beta_0, \gamma_0)$ of the zero object $(E_0,\beta_0,\gamma_0)$ of ${\bExt}(K_1,K_2)$: 
to an automorphism $(\underline{F},\underline{\Upsilon},\underline{\Phi})$ of $(E_0,\beta_0, \gamma_0)$ the canonical isomorphism associates the automorphism $(\underline{F},\underline{\Upsilon},\underline{\Phi}) + id_{(E,\beta,\gamma)}$
of $(E_0,\beta_0,\gamma_0)+ (E,\beta,\gamma) \cong (E,\beta,\gamma)$.
 Expli\-ci\-tly, ${\cExt}^0(K_1,K_2)$ consists of the couple $(f_0,f_1)$ where
\begin{itemize}
  \item $f_0:B_1 \rightarrow B_2$ is an automorphism of the trivial extension $E_0$ (i.e. $f_0 \in {\Aut}(E_0)={\Ext}^0(B_1,B_2)$), and
  \item $f_1:A_1 \rightarrow A_2$ is an automorphism of the trivial extension $T_0$ (i.e. $f_1 \in {\Aut}(T_0)={\Ext}^0(A_1,A_2)$) such that, via the isomorphism of extensions $\Theta_0: u_{2 \,*}T_0 \rightarrow u_1^*E_0$, the push-down $u_{2 \,*} f_1$ of the automorphism $f_1$ of $T_0$ is compatible with the pull-back $u_1^*f_0$ of the automorphism $f_0$ of $E_0$, i.e. $u_2 \circ f_1= f_0 \circ u_1$.
\end{itemize}
We have therefore the canonical isomorphism
\[{\cExt}^0(K_1,K_2) \cong {\Hom}_{\cK(\bS)}(K_1,K_2).\]
The group law of the category ${\bExt}(K_1,K_2)$ induces a group law on the set of isomorphism classes of objects 
of ${\bExt}(K_1,K_2)$ which we denote by ${\cExt}^1(K_1,K_2)$.

Let $K_i=[u_i:A_i \rightarrow B_i]$ (for $i=1,2,3$) be three 1-motives defined over $S$.

\begin{defn}\label{biext}
A \textbf{biextension $({\mathcal{B}}, \Psi_1, \Psi_2,\lambda)$ of $(K_1,K_2)$ by $K_3$} consists of
\begin{enumerate}
    \item a biextension $\mathcal{B}$ of $(B_1, B_2)$ by $B_3$;
    \item a trivialization 
    $$\Psi_1: A_1\times B_2 \longrightarrow (u_1,id_{B_2})^*{\mathcal{B}}$$
     of the biextension $(u_1,id_{B_2})^*{\mathcal{B}}$ of $(A_1, B_2)$ by $B_3$ obtained as pull-back of ${\mathcal{B}}$ via $(u_1,id_{B_2}): A_1 \times B_2 \rightarrow B_1 \times B_2$, and a trivialization 
     $$\Psi_2 : B_1 \times A_2 \longrightarrow (id_{B_1},u_2)^*{\mathcal{B}}$$
      of the biextension $(id_{B_1},u_2)^*{\mathcal{B}}$ of $(B_1, A_2)$ by $B_3$ obtained as pull-back of ${\mathcal{B}}$ via $(id_{B_1}, u_2): B_1 \times A_2 \rightarrow B_1 \times B_2$. These two trivializations $\Psi_1$ and $\Psi_2$ have to coincide over $A_1 \times A_2$;
    \item a trivial biextension $\mathcal{T}_1=(\mathcal{T}_1, \lambda_1)$ of $(A_1, B_2)$ by $A_3$,
    an isomorphism of biextensions 
    $$\Theta_1 : u_{3 \, *} \mathcal{T}_1 \longrightarrow (u_1,id_{B_2})^*{\mathcal{B}}$$
     between the push-down via $u_3: A_3 \rightarrow B_3$ of $\mathcal{T}_1$ and $(u_1,id_{B_2})^*{\mathcal{B}}$, a trivial biextension $\mathcal{T}_2=(\mathcal{T}_2, \lambda_2)$ of $(B_1, A_2)$ by $A_3$ and an isomorphism of biextensions 
    $$\Theta_2 : u_{3 \, *} \mathcal{T}_2 \longrightarrow (id_{B_1},u_2)^*{\mathcal{B}}$$
     between the push-down via $u_3: A_3 \rightarrow B_3$ of $\mathcal{T}_2$ and $(id_{B_1},u_2)^*{\mathcal{B}}$.
     Through the isomorphism $\Theta_1 $ the trivialization $ u_3 \circ \lambda_1$ of $u_{3 \, *} \mathcal{T}_1$ is compatible with the trivialization $\Psi_1$ of $(u_1,id_{B_2})^*{\mathcal{B}}$, and through the isomorphism $\Theta_2 $ the trivialization $ u_3 \circ \lambda_2$ of $u_{3 \, *} \mathcal{T}_2$ is compatible with the trivialization $\Psi_2$ of $(id_{B_1},u_2)^*{\mathcal{B}}$.
      The two trivializations $\lambda_1$ and $\lambda_2$ have to coincide over $A_1 \times A_2$, i.e.
     $(id_{A_1},u_2)^*\mathcal{T}_1 =(u_1,id_{A_2})^*\mathcal{T}_2$ (we will denote this biextension by $\mathcal{T}=(\mathcal{T},\lambda)$ with $\lambda$ the restriction of the trivializations $\lambda_1$ and $\lambda_2$ over $A_1 \times A_2$). Moreover, we require an isomorphism of biextensions 
     $$\Theta: u_{3 \, *} \mathcal{T} \longrightarrow (u_1,u_2)^*{\mathcal{B}}$$
     which is compatible with the isomorphisms $\Theta_1$ and $\Theta_2$ and through which 
     the trivialization $ u_3 \circ\lambda$ of $u_{3 \, *} \mathcal{T}$ is compatible with the restriction $\Psi$ of the trivializations $\Psi_1$ and $\Psi_2$ over $A_1 \times A_2$.
   \end{enumerate}
\end{defn}

Condition (3) can be rewritten as 
\par \quad (3') an morphism $\lambda: A_1 \otimes A_2 \rightarrow A_3$ such that $u_3 \circ \lambda $ is compatible with $\Psi$.\\
 
Let $K_i=[u_i:A_i \rightarrow B_i]$ and
$K'_i=[u'_i:A'_i \rightarrow B'_i]$ (for $i=1,2,3$) be 1-motives defined over $S$. Let $({\mathcal{B}},\Psi_{1}, \Psi_{2},\lambda)$ be a biextension of $(K_1,K_2)$ by $K_3$ and let $({\mathcal{B}}',\Psi'_{1}, \Psi'_{2},\lambda')$ be a biextension of
 $(K'_1,K'_2)$ by $K'_3$.

\begin{defn}
A \textbf{morphism of biextensions}
\[(\underline{F},\underline{\Upsilon}_1,\underline{\Upsilon}_2,\underline{\Phi}):
({\mathcal{B}},\Psi_{1}, \Psi_{2},\lambda)
\longrightarrow ({\mathcal{B}}',\Psi'_{1}, \Psi'_{2},\lambda')\]
\pn consists of
\begin{enumerate}
    \item a morphism $\underline{F}=(F,f_1,f_2,f_3):{\mathcal{B}} \rightarrow {\mathcal{B}}'$ from the biextension ${\mathcal{B}}$ to the biextension ${\mathcal{B}}'$. In particular, $F: {\mathcal{B}} \rightarrow {\mathcal{B}}'$ is a morphism of the sheaves underlying ${\mathcal{B}}$ and ${\mathcal{B}}'$, and
        \[ f_1:B_1 \longrightarrow B'_1 \qquad f_2:B_2 \longrightarrow B'_2 \qquad f_3:B_3 \longrightarrow B'_3\]
    \pn are morphisms abelian sheaves on $\bS$.
    \item a morphism of biextensions
    \[\underline{\Upsilon}_1 = (\Upsilon_1,g_1,f_2,f_3):(u_1, id_{B_2})^* {\mathcal{B}} \longrightarrow (u'_1, id_{B'_2})^* {\mathcal{B}}'\]
    \pn compatible with the morphism $\underline{F}=(F,f_1,f_2,f_3)$ and with the trivializations $\Psi_{1}$ and $\Psi'_{1}$, and  a morphism of biextensions
    \[\underline{\Upsilon}_2 = (\Upsilon_2,f_1,g_2,f_3):(id_{B_1},u_2)^* {\mathcal{B}} \longrightarrow (id_{B'_1},u'_2)^* {\mathcal{B}}'\]
    \pn compatible with the morphism $\underline{F}=(F,f_1,f_2,f_3)$ and with the trivializations $\Psi_{2}$ and $\Psi'_{2}$. In particular, $\Upsilon_1: (u_1, id_{B_2})^* {\mathcal{B}} \rightarrow (u'_1, id_{B'_2})^* {\mathcal{B}}'$ is a morphism of the sheaves underlying $(u_1, id_{B_2})^*{\mathcal{B}}$ and $(u'_1, id_{B'_2})^* {\mathcal{B}}'$, $\Upsilon_2: (id_{B_1},u_2)^* {\mathcal{B}} \rightarrow (id_{B'_1},u'_2)^* {\mathcal{B}}'$ is a morphism of the sheaves underlying $ (id_{B_1},u_2)^* {\mathcal{B}} $ and $(id_{B'_1},u'_2)^* {\mathcal{B}}'$, and
    \[g_1: A_1 \longrightarrow A'_1 \qquad g_2: A_2 \longrightarrow A'_2  \]
    \pn are morphisms abelian sheaves on $\bS$.
    \pn By pull-back, the two morphisms $\underline{\Upsilon}_1=(\Upsilon_1,g_1,f_2,f_3)$ and $\underline{\Upsilon}_2=(\Upsilon_2,f_1,g_2,f_3)$ define a morphism of biextensions
    $$\underline{\Upsilon}=(\Upsilon,g_1,g_2,f_3): (u_1,u_2)^* {\mathcal{B}} \longrightarrow (u'_1, u'_2)^* {\mathcal{B}}'$$
     compatible with the morphism $\underline{F}=(F,f_1,f_2,f_3)$ and with the trivializations $\Psi$ (restriction of $\Psi_1$ and $\Psi_2$ over $A_1 \times A_2$) and $\Psi'$ (restriction of $\Psi'_1$ and $\Psi'_2$ over $A'_1 \times A'_2$).
    \item a morphism of biextensions
    \[\underline{\Phi}_1 = (\Phi_1,g_1,f_2,g_3): {\mathcal{T}}_1 \longrightarrow  {\mathcal{T}}'_1\]
    \pn compatible with the morphism $\underline{\Upsilon}_1=(\Upsilon,g_1,f_2,f_3)$ and with the trivializations $\lambda_{1}$ and $\lambda'_{1}$, and  a morphism of biextensions
    \[\underline{\Phi}_2 = (\Phi_2,f_1,g_2,g_3): {\mathcal{T}}_2 \longrightarrow {\mathcal{T}}'_2\]
    \pn compatible with the morphism $\underline{\Upsilon}_2 = (\Upsilon_2,f_1,g_2,f_3)$ and with the trivializations $\lambda_{2}$ and $\lambda'_{2}$. In particular, $\Phi_1: {\mathcal{T}}_1 \rightarrow {\mathcal{T}}'_1$ is a morphism of the sheaves underlying ${\mathcal{T}}_1$ and ${\mathcal{T}}'_1$, $\Phi_2: {\mathcal{T}}_2 \rightarrow {\mathcal{T}}'_2$ is a morphism of the sheaves underlying $ {\mathcal{T}}_2 $ and ${\mathcal{T}}'_2$, and
    \[g_3: A_3 \longrightarrow A'_3   \]
    \pn is an morphism abelian sheaves on $\bS$.
    \pn By pull-back, the two morphisms $\underline{\Phi}_1 = (\Phi_1,g_1,f_2,g_3)$ and $\underline{\Phi}_2 = (\Phi_2,f_1,g_2,g_3)$ define a morphism of biextensions
    $$\underline{\Phi} = (\Phi,g_1,g_2,g_3):  {\mathcal{T}} \longrightarrow {\mathcal{T}}'$$
     compatible with the morphism $\underline{\Upsilon}=(\Upsilon,g_1,g_2,f_3)$ and with the trivializations $\lambda$ (restriction of $\lambda_1$ and $\lambda_2$ over $A_1 \times A_2$) and $\lambda'$
    (restriction of $\lambda'_1$ and $\lambda'_2$ over $A_1 \times A_2$).
\end{enumerate}
\end{defn}

Condition (3) can be rewritten as 
\par \quad (3') an morphism $g_3:A_3 \rightarrow A'_3$ abelian sheaves on $\bS$ compatible with $u_3$ and $u'_3$ (i.e. $u'_3 \circ g_3 =f_3 \circ u_3$) and such that
    \[ \lambda' \circ (g_1 \times g_2)= g_3 \circ \lambda.\]
Explicitly, the compatibility of $\underline{\Upsilon}_1$ with $\underline{F}$, $\Psi_1$ and $\Psi'_1$ means that the following diagram is commutative:
\[\begin{array}{ccccc}
      A_1 \times B_2& \stackrel{\Psi_1}{\longrightarrow}  &   (u_1, id_{B_2})^* {\mathcal{B}} & \longrightarrow & {\mathcal{B}}\\
  \scriptstyle{g_1 \times f_2}    \downarrow & & \scriptstyle{\Upsilon_1}    \downarrow & & \downarrow \scriptstyle{F}\\
       A'_1 \times B'_2 & \stackrel{\Psi'_1}{\longrightarrow}  &  (u'_1, id_{B'_2})^* {\mathcal{B}}' &  \longrightarrow  &  {\mathcal{B}}' . \\
      \end{array}
\]
The compatibility of $\underline{\Upsilon}_2$ with $\underline{F}$, $\Psi_2$ and $\Psi'_2$ means that the following diagram is commutative:
\[\begin{array}{ccccc}
      B_1 \times A_2& \stackrel{\Psi_2}{\longrightarrow}  &   (id_{B_1},u_2)^* {\mathcal{B}} & \longrightarrow & {\mathcal{B}}\\
  \scriptstyle{f_1 \times g_2}    \downarrow & & \scriptstyle{\Upsilon_2}    \downarrow & & \downarrow \scriptstyle{F}\\
       B'_1 \times A'_2 & \stackrel{\Psi'_2}{\longrightarrow}  &  (id_{B'_1},u'_2)^* {\mathcal{B}}' &  \longrightarrow  &  {\mathcal{B}}' . \\
      \end{array}
\]
The compatibility of $\underline{\Upsilon}$ with $\underline{F}$, $\Psi$ and $\Psi'$ means that the following diagram is commutative:
\[\begin{array}{ccccc}
      A_1 \times A_2& \stackrel{\Psi}{\longrightarrow}  &   (u_1,u_2)^* {\mathcal{B}} & \longrightarrow & {\mathcal{B}}\\
  \scriptstyle{g_1 \times g_2}    \downarrow & & \scriptstyle{\Upsilon}    \downarrow & & \downarrow \scriptstyle{F}\\
       A'_1 \times A'_2 & \stackrel{\Psi'}{\longrightarrow}  &  (u'_1,u'_2)^* {\mathcal{B}}' &  \longrightarrow  &  {\mathcal{B}}' . \\
      \end{array}
\]
The compatibility of $\underline{\Phi}_1$ with $\underline{\Upsilon}_1$, $\lambda_1$ and $\lambda'_1$ means that the following diagram is commutative:
\[\begin{array}{ccccccc}
      A_1 \times B_2& \stackrel{\lambda_1}{\longrightarrow}  &  {\mathcal{T}}_1 & \longrightarrow &u_{3\,*} {\mathcal{T}}_1 & \stackrel{\Theta_1} \cong & (u_1, id_{B_2})^*{\mathcal{B}}\\
  \scriptstyle{g_1 \times f_2}    \downarrow & & \scriptstyle{\Phi_1}    \downarrow & &&& \downarrow \scriptstyle{\Upsilon_1}\\
       A'_1 \times B'_2 & \stackrel{\lambda'_1}{\longrightarrow}  &  {\mathcal{T}}'_1 & \longrightarrow &u'_{3\,*} {\mathcal{T}}'_1 & \stackrel{\Theta'_1} \cong &  (u'_1, id_{B'_2})^*{\mathcal{B}}' . \\
      \end{array}
\]
The compatibility of $\underline{\Phi}_2$ with $\underline{\Upsilon}_2$, $\lambda_2$ and $\lambda'_2$ means that the following diagram is commutative:
\[\begin{array}{ccccccc}
      B_1 \times A_2& \stackrel{\lambda_2}{\longrightarrow}  &    {\mathcal{T}}_2 & \longrightarrow &u_{3\,*} {\mathcal{T}}_2 & \stackrel{\Theta_2} \cong & (id_{B_1},u_2)^*{\mathcal{B}}\\
  \scriptstyle{f_1 \times g_2}    \downarrow & & \scriptstyle{\Phi_2}    \downarrow & &&& \downarrow \scriptstyle{\Upsilon_2}\\
       B'_1 \times A'_2 & \stackrel{\lambda'_2}{\longrightarrow}  &   {\mathcal{T}}'_2 & \longrightarrow &u'_{3\,*} {\mathcal{T}}'_2 & \stackrel{\Theta'_2} \cong & (id_{B'_1},u'_2)^* {\mathcal{B}}' . \\
      \end{array}
\]
Finally, the compatibility of $\underline{\Phi}$ with $\underline{\Upsilon}$, $\lambda$ and $\lambda'$ means that the following diagram is commutative:
\[\begin{array}{ccccccc}
      A_1 \times A_2& \stackrel{\lambda}{\longrightarrow}  &  {\mathcal{T}} & \longrightarrow &u_{3\,*} {\mathcal{T}} & \stackrel{\Theta} \cong &  (u_1,u_2)^* {\mathcal{B}}\\
  \scriptstyle{g_1 \times g_2}    \downarrow & & \scriptstyle{\Phi}    \downarrow & && & \downarrow \scriptstyle{\Upsilon}\\
       A'_1 \times A'_2 & \stackrel{\lambda'}{\longrightarrow}  &   {\mathcal{T}}' & \longrightarrow &u'_{3\,*} {\mathcal{T}}' & \stackrel{\Theta'} \cong & (u'_1,u'_2)^*{\mathcal{B}}' . \\
      \end{array}
\]
We denote by ${\bBiext}(K_1,K_2;K_3)$ the category of biextensions
of $(K_1,K_2)$ by $K_3$. 
 The Baer sum of extensions defines a group law for the objects of this category which is therefore 
a strictly commutative Picard category (see~\cite{SGA7} Expos\'e VII 2.5).  The zero object $({\mathcal{B}}_0, \Psi_{01}, \Psi_{02},\lambda_0)$ of ${\bBiext}(K_1,K_2;K_3)$ with respect to this group law consists of
\begin{itemize}
  \item the trivial biextension ${\mathcal{B}}_0=B_1 \times B_2 \times B_3$ of $(B_1,B_2)$ by $B_3$, i.e. the zero object of ${\bBiext}(B_1,B_2;B_3)$, and
  \item the trivialization $\Psi_{01}=(id_{A_1},id_{B_2},0)$ (resp.
 $\Psi_{02}=(id_{B_1},id_{A_2},0)$) of the biextension $(u_1,id_{B_2})^*{\mathcal{B}}_0=A_1 \times B_2 \times B_3$ of $(A_1,B_2)$ by $B_3$ (resp. of the biextension $(id_{B_1},u_2)^*{\mathcal{B}}_0=B_1 \times A_2 \times B_3$ of $(B_1\times A_2)$ by $B_3$),
   \item the trivial biextension $\mathcal{T}_{10}$ of $(A_1, B_2)$ by $A_3$ (i.e. ${\mathcal{T}}_{10}=({\mathcal{T}}_{10}, \lambda_{10})$ with ${\mathcal{T}}_{10}= A_1 \times B_2 \times A_3$ and $\lambda_{10}=(id_{A_1},id_{B_2},0)$), the isomorphism of biextensions $\Theta_{10}=(id_{A_1}, id_{B_2},id_{B_3}):u_{3 \,*} 
 \mathcal{T}_{10} \rightarrow (u_1,id_{B_2})^*{\mathcal{B}}_0$,
 the trivial biextension $\mathcal{T}_{01}$ of $(B_1, A_2)$ by $A_3$ (i.e. $\mathcal{T}_{01}=(\mathcal{T}_{01}, \lambda_{01})$ with ${\mathcal{T}}_{10}= B_1 \times A_2 \times A_3$ and $\lambda_{01}=(id_{B_1},id_{A_2},0))$ and the isomorphism of biextensions $\Theta_{01}=(id_{B_1}, id_{A_2},id_{B_3}):u_{3 \,*} 
 \mathcal{T}_{10} \rightarrow (id_{B_1},u_2)^*{\mathcal{B}}_0$. In particular the restriction of $\lambda_{10}$ and $\lambda_{01}$ over $A_1 \times A_2$ is $\lambda_{0}=(id_{A_1},id_{A_2},0)$
\end{itemize}
We denote by ${\Biext}^0(K_1,K_2;K_3)$ the group of automorphisms of any object of ${\bBiext}(K_1,K_2;K_3)$. It is canonically isomorphic to the group of 
automorphisms of the zero object $({\mathcal{B}}_0, \Psi_{01}, \Psi_{02},\lambda_0)$. Explicitly, ${\Biext}^0(K_1,K_2;K_3)$ consists of the couple $(f_0,(f_{10},f_{01}))$ where
\begin{itemize}
  \item $f_{0}:B_1 \otimes B_2 \rightarrow B_3$ is an automorphism of the trivial biextension ${\mathcal{B}}_0$ (i.e. $f_0 \in {\Biext}^0(B_1,B_2;B_3)$), and
  \item $f_{10}:A_1 \otimes B_2 \rightarrow A_3$ is an automorphism of the trivial biextension ${\mathcal{T}}_{10}$ (i.e. $f_{10} \in {\Biext}^0(A_1,B_2;A_3)$) and $f_{01}:B_1 \otimes A_2 \rightarrow A_3$ is an automorphism of the trivial biextension ${\mathcal{T}}_{01}$ (i.e. $f_{01} \in {\Biext}^0(B_1,A_2;A_3)$) 
  such that, via the isomorphisms of biextensions 
   $\Theta_{10}: u_{3 \, *} \mathcal{T}_{10} \rightarrow (u_1,id_{B_2})^*{\mathcal{B}}_0$
    and 
    $\Theta_{01}: u_{3 \, *}\mathcal{T}_{01} \rightarrow (id_{B_1},u_2)^*{\mathcal{B}}_0$,
   the push-down $u_{3 \, *} f_{10}$ of $ f_{10}$ is compatible with the pull-back $(u_1,id_{B_2})^*f_0$ of $f_0$, and the push-down $u_{3 \, *} f_{01}$ of $f_{01}$ is compatible with the pull-back $(id_{B_1},u_2)^*f_0$ of $f_0$, i.e. such that the following diagram commute
  \[    \xymatrix{
      A_1 \otimes B_2 + B_1 \otimes A_2  \qquad \qquad \ar[r]^{(u_1,id)+(id,u_2)} \ar[d]_{f_{10}+f_{01}} & \quad B_1 \otimes B_2 \ar[d]^{f_{0}}\\
       A_3 \ar[r]_{u_3}   & B_3 .
}\]  
\end{itemize}
We have therefore the canonical isomorphism
\[{\Biext}^0(K_1,K_2;K_3) \cong {\Hom}_{\cK(\bS)}(K_1{\buildrel {\scriptscriptstyle \LL}
 \over \otimes}\,K_2, K_3).\]
The group law of the category ${\bBiext}(K_1,K_2;K_3)$ induces a group law on the set of isomorphism classes of objects of ${\bBiext}(K_1,K_2;K_3)$, 
that we denote by ${\Biext}^1(K_1,K_2;K_3)$.

\begin{rem}
 According to the above geometrical definitions of extensions and biextensions of 1-motives, we have the following equivalence of categories
 \[ {\bBiext}(K_1,[0 \rightarrow {\ZZ}];K_3) \simeq {\bExt}(K_1,K_3).\]
Moreover we have also the following isomorphisms
\[ {\Biext}^i(K_1,[{\ZZ} \rightarrow 0];K_3) = \left\{
        \begin{array}{ll}
          {\Hom}(B_1,A_3), & \hbox{$i=0$;} \\
          {\Hom}(K_1,K_3), & \hbox{$i=1$.}
        \end{array}
\right.\]
 Note that we get the same results applying the homological interpretation of biextensions furnished by our main Theorem~\ref{mainthm}.
\end{rem}

%-----------------------------------------------------------------------------
\section{Review on strictly commutative Picard stacks}

Let {\bS} be a site. For the notions of {\bS}-pre-stack, {\bS}-stack and morphisms of {\bS}-stacks we refer to \cite{G} Chapter II 1.2.

A \textbf{strictly commutative Picard {\bS}-stack} is an {\bS}-stack of groupoids $\pic$ endowed with a functor $ +: \pic \times_{\bS} \pic \rightarrow \pic, ~~(a,b) \mapsto a+b$, and two 
  natural isomorphisms of associativity $\sigma$ and of commutativity $\tau$, which are described by the functorial isomorphisms
\begin{eqnarray}
\label{ass}  \sigma_{a,b,c} &:& (a + b) + c \stackrel{\cong}{\longrightarrow} a + ( b + c) \qquad \forall~ a,b,c \in \pic, \\
\label{com} \tau_{a,b} &:& a + b \stackrel{\cong}{\longrightarrow} b + a \qquad \forall~ a,b \in \pic;
\end{eqnarray}
such that for any object $U$ of $\bS$, $(\pic (U),+,\sigma, \tau)$ is a strictly commutative Picard category (i.e. it is possible to make the sum of two objects of $\pic (U)$ and this sum is associative and commutative, see \cite{SGA4} 1.4.2 for more details). Here "strictly" means that $\tau_{a,a}$ is the identity 
for all $a \in \pic$.
Any strictly commutative Picard {\bS}-stack admits a global neutral object $e$ and the sheaf of automorphisms of the neutral object ${\uAut}(e)$ is abelian.

Let $\pic$ and $\qic$ be two strictly commutative Picard {\bS}-stacks.
 An \textbf{additive functor} $(F,\sum):\pic \rightarrow \qic $
between strictly commutative Picard {\bS}-stacks is a morphism of {\bS}-stacks (i.e. a cartesian {\bS}-functor, see \cite{G} Chapter I 1.1) endowed with a natural isomorphism $\sum$ which is described by the functorial isomorphisms
\[\sum_{a,b}:F(a+b) \stackrel{\cong}{\longrightarrow} F(a)+F(b)
\qquad \forall~ a,b \in \pic \]
and which is compatible with the natural isomorphisms $\sigma$ and $\tau$ of
$\pic$ and $\qic$. A \textbf{morphism of additive functors} $u:(F,\sum) \rightarrow (F',\sum') $ is an {\bS}-morphism of cartesian {\bS}-functors (see \cite{G} Chapter I 1.1) which is compatible with the natural isomorphisms $\sum$ and $\sum'$ of $F$ and $F'$ respectively.

An \textbf{equivalence of strictly commutative Picard {\bS}-stacks} between $\pic$ and $\qic$ is an additive functor $(F,\sum):\pic \rightarrow \qic$ with $F$ an equivalence of {\bS}-stacks. Two strictly commutative Picard {\bS}-stacks are \textbf{equivalent as strictly commutative Picard {\bS}-stacks} if there exists an equivalence of strictly commutative Picard {\bS}-stacks between them.

To any strictly commutative Picard {\bS}-stack $\pic$, we associate the sheaffification $\pi_0(\pic)$ of the pre-sheaf which associates to each object $U$ of $\bS$ the group of isomorphism classes of objects of $\pic(U)$, and
 the sheaf $\pi_1(\pic)$ of automorphisms ${\uAut}(e)$ of the neutral object of $\pic$.

In~\cite{SGA4} \S 1.4 Deligne associates to each complex
$K$ of $\cK^{[-1,0]}(\bS)$ a strictly commutative Picard {\bS}-stack $st(K)$ and to each morphism of complexes $g: K \rightarrow L$ an additive functor 
$st(g):st(K) \rightarrow st(L)$. 
Moreover, if ${\bPicard}(\bS)$ denotes the category whose objects are small strictly commutative Picard $\bS$-stacks and whose arrows are isomorphism classes of additive functors, Deligne proves the following equivalence of category
\begin{eqnarray}
\label{st} st: \cD^{[-1,0]}(\bS) &\longrightarrow & {\bPicard}(\bS) \\
 \nonumber  K & \mapsto & st(K) \\
 \nonumber K \stackrel{f}{\rightarrow} L & \mapsto &  st(K) \stackrel{st(f)}{\rightarrow} st(L).
\end{eqnarray}
constructing explicitly the inverse equivalence of $st$, that we denote by $[\,\,]$.

\begin{ex}\label{ex:pic} Let $\pic, \qic$ and $\gic$ be three strictly commutative Picard {\bS}-stacks. \\
I) Let 
$${\HOM}(\pic,\qic)$$
be the strictly commutative Picard {\bS}-stack defined as followed:
 for any object $U$ of $\bS$, the objects of the category ${\HOM}(\pic,\qic)(U)$ are additive functors from ${\pic}_{|U}$ to ${\qic}_{|U}$ and its arrows are morphisms of additive functors.
According~\cite{SGA4} 1.4.18 we have the equality 
$[{\HOM}(\pic,\qic)] = \tau_{\leq 0}{\R}{\Hom}\big([\pic],[\qic]\big)$ in the derived category $\cD^{[-1,0]}(\bS).$\\
II) A \textbf{biadditive functor} $(F,l,r):\pic \times \qic \rightarrow \gic $
 is a morphism of {\bS}-stacks $F:\pic \times \qic \rightarrow \gic $ endowed with two natural isomorphisms, which are described by the functorial isomorphisms
\begin{eqnarray}
 \nonumber l_{a,b,c}:F(a+b,c) &\stackrel{\cong}{\longrightarrow}& F(a,c)+F(b,c) \qquad \forall~ a,b \in \pic, \; \forall~ c \in \qic  \\
 \nonumber r_{a,c,d}:F(a,c+d) &\stackrel{\cong}{\longrightarrow}& F(a,c)+F(a,d)
 \qquad \forall~ a \in \pic, \; \forall~ c,d \in \qic,
\end{eqnarray}
 such that 
\begin{itemize}
	\item for any fixed $a \in \pic$, $F(a,-)$ is compatible with the natural isomorphisms $\sigma$ and $\tau$ of $\pic$ and $\gic$,
       \item for any fixed $c \in \qic$, $F(-,c)$ is compatible with the natural isomorphisms $\sigma$ and $\tau$ of $\qic$ and $\gic$,
        \item for any fixed $a,b\in \pic$ and $c,d \in \qic$ is the following diagram commute 
  \[    \xymatrix{
     F(a+b,c+d)\ar[r]^r  \ar[d]_l& F(a+b,c)+F(a+b,d) \ar[r]^{l+l \qquad}& F(a,c)+F(b,c)+F(a,d)+F(b,d)   \\
   F(a,c+d)+F(b,c+d) \ar[rr]_{r+r}  &  & F(a,c)+F(a,d)+F(b,c)+F(b,d) \ar[u]_{id_\gic +\tau+ id_\gic} .
}\]  
\end{itemize}
A \textbf{morphism of biadditive functors $\alpha:(F,l,r) \Rightarrow (F',l',r') $} is a morphism of morphisms of {\bS}-stacks $\alpha:F \Rightarrow F' $ which is compatible with the natural isomorphisms $l,r$ and $l',r$ of $F$ and $F'$ respectively.
 Let 
 $${\HOM}(\pic,\qic;\gic)$$
be the strictly commutative Picard {\bS}-stack defined as followed:
 for any object $U$ of $\bS$, the objects of the category ${\HOM}(\pic,\qic;\gic)(U)$ are biadditive functors from ${\pic}_{|U} \times {\qic}_{|U}$ to ${\gic}_{|U}$ and its arrows are morphisms of biadditive functors.\\
 III) Let 
 $$\pic \otimes \qic$$
  be the strictly commutative Picard $\bS$-stack endowed with a biadditive functor $\otimes :\pic \times \qic \rightarrow \pic \otimes \qic$ such that for any strictly commutative Picard $\bS$-stack $\gic$, the biadditive functor $\otimes$ defines the following equivalence of strictly commutative Picard $\bS$-stacks: 
\begin{equation}
{\HOM}(\pic \otimes \qic, \gic) \cong  {\HOM}(\pic,\qic;\gic).
\label{eq:tensorlinear}
\end{equation}
According to \cite{SGA4} 1.4.20, in the derived category $\cD^{[-1,0]}(\bS)$ we have the equality 
$[\pic\otimes \qic] = \tau_{\geq -1} ([\pic] \otimes^\LL[\qic]) $.
\end{ex}

By \S 2 \cite{Be10} we have the following operations on strictly commutative Picard $\bS$-stacks:\\
(1) The \textbf{product} of two strictly commutative Picard $\bS$-stacks $\pic$ and $\qic$ is the strictly commutative Picard $\bS$-stack $\pic \times \qic$ defined as followed:
\begin{itemize}
  \item for any object $U$ of $\bS$, an object of the category $\pic \times \qic(U)$ is a pair $(p,q)$ of objects with $p$ an object of $\pic(U)$ and $q$ an object of $\qic(U)$;
  \item for any object $U$ of $\bS$, if $(p,q)$ and $(p',q')$ are two objects of 
  $\pic \times \qic(U)$, an arrow of $\pic \times \qic(U)$ from $(p,q)$ to $(p',q')$ is a pair $(f,g)$ of arrows with $f:p \rightarrow  p' $ an arrow of $\pic(U)$ and $g:q \rightarrow  q' $ an arrow of $\qic(U)$.
\end{itemize}
(2) Let $G:\pic \rightarrow \qic$ and $F:\pic' \rightarrow \qic$ be additive functors between strictly commutative Picard $\bS$-stacks. The \textbf{fibered product} of $\pic$ and $\pic'$ over $\qic$ via $F$ and $G$ is the strictly commutative Picard $\bS$-stack $\pic \times_\qic \pic'$ defined as followed: 
\begin{itemize}
  \item for any object $U$ of {\bS}, the objects of the category $(\pic \times_\qic \pic')(U)$ are triplets $(p,p',f)$ where $p$ is an object of $\pic(U)$, $p' $ is an object of $\pic'(U)$ and $f:G(p) \stackrel{\cong}{\rightarrow}F(p')$ is an isomorphism of $\qic(U)$ between $G(p)$ and $F(p')$;
  \item for any object $U$ of {\bS}, if $(p_1,p'_1,f)$ and $(p_2,p'_2,g)$ are two objects of
  $(\pic \times_\qic \pic')(U)$, an arrow of $(\pic \times_\qic \pic')(U)$ 
 from  $(p_1,p'_1,f)$ to $(p_2,p'_2,g)$ is a pair $(f,g)$ of arrows with
  $\alpha:p_1 \rightarrow p_2$ of arrow of $\pic(U)$ and $\beta:p'_1 \rightarrow p'_2$ an arrow of $\pic'(U)$ such that $ g \circ G(\alpha) = F(\beta) \circ f$.
\end{itemize}
The fibered product $\pic \times_\qic \pic'$ is also called the \textbf{pull-back} $F^*\pic$ of $\pic$ via $F:\pic' \rightarrow \qic$ or the \textbf{pull-back} $G^*\pic'$ of $\pic'$ via $G:\pic \rightarrow \qic$.\\
(3) Let $G:\qic \rightarrow \pic$ and $F:\qic \rightarrow \pic'$ be additive functors between strictly commutative Picard $\bS$-stacks. The \textbf{fibered sum} of $\pic$ and $\pic'$ under $\qic$ via $F$ and $G$ is the strictly commutative Picard $\bS$-stack $\pic +^\qic \pic'$ generated by the following strictly commutative Picard $\bS$-pre-stack $\dic$:
\begin{itemize}
  \item for any object $U$ of {\bS}, the objects of the category $\dic(U)$ are pairs $(p,p')$ with $p$ an object of $\pic(U)$ and  $p'$ an object of $\pic'(U)$;
  \item for any object $U$ of {\bS}, if $(p_1,p'_1)$ and $(p_2,p'_2)$ are two objects of $\dic(U)$, an arrow of $\dic (U)$ from $(p_1,p'_1)$ to $(p_2,p'_2)$ is an
 equivalence class of triplets  
  $(q,\alpha,\beta)$ with $q$ an object of $\qic (U)$, $\alpha: p_1 + G(q) \rightarrow p_2$ an arrow of $\pic(U)$ and
 $\beta: p'_1 + F(q) \rightarrow p'_2$ an arrow of $\pic'(U)$. Two triplets $(q_1,\alpha_1,\beta_1)$ and $(q_2,\alpha_2,\beta_2)$
are equivalent it there is an arrow $\gamma: q_1 \rightarrow q_2$ in $\qic (U)$ such that $\alpha_2 \circ (id +G(\gamma)) = \alpha_1$ and
 $ (F(\gamma) + id) \circ \beta_1 =\beta_2 $.  
\end{itemize}
The fibered sum $\pic +^\qic \pic'$ is also called the \textbf{push-down} $F_*\pic$ of $\pic$ via $F:\qic \rightarrow \pic'$ or the \textbf{push-down} $G_*\pic'$ of $\pic'$ via $G:\qic \rightarrow \pic$.

We have analogous operations on complexes of $\cK^{[-1,0]}(\bS)$:\\
(1) The \textbf{product} of two complexes $P=[d^P:P^{-1} \rightarrow P^0]$ and $ Q=[d^Q: Q^{-1} \rightarrow Q^0]$ of $\cK^{[-1,0]}(\bS)$ is the complex $P+Q=[(d^P,d^Q):P^{-1} + Q^{-1}\rightarrow P^0 +Q^0]$. Via the equivalence of category (\ref{st}) we have that $st(P+ Q) = st(P) \times st(Q)$.\\
(2) Let $P=[d^P:P^{-1} \rightarrow P^0], Q=[d^Q: Q^{-1} \rightarrow Q^0]$ and $G=[d^G: G^{-1} \rightarrow G^0]$ be complexes of $\cK^{[-1,0]}(\bS)$ and let $f:P \rightarrow G$ and $g:Q \rightarrow G$ be morphisms of complexes. The \textbf{fibered product} $P \times_G Q$ of $P$ and $Q$ over $G$
 is the complex $[d_P \times_{d_G} d_Q: P^{-1} \times_{G^{-1}} Q^{-1} \rightarrow P^{0} \times_{G^0} Q^{0}]$,
where for $i=-1,0$ the abelian sheaf $P^{i} \times_{G^i} Q^{i}$ is the fibered product of $P^i$ and of $Q^i$ over $G^i$ and the morphism of abelian sheaves $d_P \times_{d_G}d_Q $ is given by the universal property of the fibered product 
$P^{0} \times_{G^0} Q^{0}$. The fibered product $P \times_G Q$ is also called the \textbf{pull-back} $g^*P$ of $P$ via $g:Q \rightarrow G$ or the \textbf{pull-back} $f^*Q$ of $Q$ via $f:P \rightarrow G$.
Remark that $st(P \times_G Q) = st(P) \times_{st(G)} st(Q)$ via the equivalence of category (\ref{st}).\\
(3) Let $P=[d^P:P^{-1} \rightarrow P^0], Q=[d^Q: Q^{-1} \rightarrow Q^0]$ and $G=[d^G: G^{-1} \rightarrow G^0]$ be complexes of $\cK^{[-1,0]}(\bS)$ and let $f:G \rightarrow P$ and $g:G \rightarrow Q$ be morphisms of complexes. The \textbf{fibered sum} $P +^G Q$ of $P$ and $Q$ under $G$
 is the complex $[d_P +^{d_G} d_Q: P^{-1} +^{G^{-1}} Q^{-1} \rightarrow P^{0} +^{G^0} Q^{0}]$,
where for $i=-1,0$ the abelian sheaf $P^{i} +^{G^i} Q^{i}$ is the fibered sum of $P^i$ and of $Q^i$ under $G^i$ and the morphism of abelian sheaves $d_P +^{d_G}d_Q $ is given by the universal property of the fibered sum 
$P^{-1} +^{G^{-1}} Q^{-1}$. The fibered sum $P +^G Q$ is also called the \textbf{push-down} $g_*P$ of $P$ via $g:G \rightarrow Q$ or the \textbf{push-down} $f_*Q$ of $Q$ via $f:G \rightarrow P$.
We have $st(P +^G Q) = st(P) +^{st(G)} st(Q)$ via the equivalence of category (\ref{st}).\\

If $\pic$ and $\gic$ are strictly commutative Picard $\bS$-stacks, by \S 3 \cite{Be10} an \textbf{extension}  $\eic=(\eic,I,J) $ of $\pic$ by $\gic$ consists of a strictly commutative Picard $\bS$-stack $\eic$, two additive functors $I:\gic \rightarrow \eic$ and $ J:\eic \rightarrow \pic$, and an isomorphism of additive functors $J \circ I \cong 0,$ such that the following equivalent conditions are satisfied:
   
   	(a) $\pi_0(J): \pi_0(\eic) \rightarrow \pi_0(\pic)$ is surjective and $I$ induces an equivalence of strictly commutative Picard $\bS$-stacks between $\gic$ and $\ker(J);$
     
      (b) $\pi_1(I): \pi_1(\gic) \rightarrow \pi_1(\eic)$ is injective and $J$ induces an equivalence of strictly commutative Picard $\bS$-stacks between $\coker(I)$ and $\pic$. 
      
In terms of complexes of $\cK^{[-1,0]}(\bS)$, an \textbf{extension} $E=(E,i,j)$ of $P$ by $G$
consists of a complex $E$ of $\cK^{[-1,0]}(\bS)$, two morphisms of complexes  $i:G \rightarrow E$ and $ j:E \rightarrow P$ of $\cK^{[-1,0]}(\bS)$, and an homotopy between $j \circ i$ and $0$,
such that the following equivalent conditions are satisfied:

	(a) ${\h}^{0}(j): {\h}^{0}(E) \rightarrow {\h}^{0}(P)$ is surjective and $i$ induces a quasi-iso\-mor\-phism between $G$ and $ \tau_{\leq 0} (MC(j)[-1])$;
 
      (b) ${\h}^{-1}(i): {\h}^{-1}(G) \rightarrow {\h}^{-1}(E)$ is injective and $j$ induces a quasi-iso\-mor\-phism between $ \tau_{\geq -1} MC(i)$ and $P$.\\

As recalled in the introduction we can see 1-motives as complexes of abelian sheaves on $\bS$ concentrated in two consecutive degrees. Hence via (\ref{st}) to each 1-motives is associated a strictly commutative Picard $\bS$-stack and in particular, we can apply all what we have recalled in this section to 1-motives. Moreover, since 
 a short exact sequence in $\cK^{[-1,0]}(\bS)$ is an extension of complexes in the above sense 
 (see \cite{Be10} Remark 3.6.), extensions of 1-motives are also extensions of complexes in the above sense, i.e. they furnishes extensions of strictly commutative Picard $\bS$-stacks (see Remark \ref{rem:ext-extcomplex}).  

%-------------------------------------------------------------------------

\section{Proof of theorem 0.1 (\textbf{b})}

\emph{Proof of Theorem \ref{mainthm}} \textbf{b}.
Via the equivalence of category (\ref{st}),
to the trivial biextension of $(K_1,K_2)$ by $K_3$ corresponds the trivial biextension ${\bic}_0= st(K_3) \times st(K_1) \times st(K_2)$ of $(st(K_1),st(K_2))$ by $st(K_3)$ (see \cite{Be11} Definition 5.1). In particular ${\bic}_0$ is a Picard stack and so the group of isomorphism classes of arrows from ${\bic}_0$ to itself is the cohomology group ${\h}^0([{\HOM}({\bic}_0,{\bic}_0)])$, where ${\HOM}({\bic}_0,{\bic}_0)$ is the strictly commutative Picard stack of additive functors from ${\bic}_0$ to itself. Therefore, in order to compute ${\Biext}^0(K_1,K_2;K_3)$ it is enough to compute the complex $[{\HOM}({\bic}_0,{\bic}_0)]$.\\
Let $F: {\bic}_0\rightarrow {\bic}_0$ be an additive functor.
Since $F$ is first of all an arrow from the $st(K_3)$-torsor over $st(K_1) \times st(K_2)$ underlying ${\bic}_0$ to itself,  $F$ is given by the formula
$$F(b)= b + IF'J (b) \qquad \forall \; b \in {\bic}_0$$
where $F': st(K_1) \times st(K_2) \rightarrow st(K_3)$ is an additive functor and $J: {\bic}_0 \rightarrow st(K_1) \times st(K_2)$ and $I: st(K_3) \rightarrow {\bic}_0$ are the additive functors underlying the structure of $st(K_3)$-torsor over $st(K_1) \times st(K_2)$ of ${\bic}_0$.
Now $F: {\bic}_0 \rightarrow {\bic}_0$ must be compatible with the structures of extension of $st(K_2)_{st(K_1)}$
by $st(K_3)_{st(K_1)}$ and of extension of $st(K_1)_{st(K_2)}$ by $st(K_3)_{st(K_2)}$ underlying ${\bic}_0$,
and so $F': st(K_1) \times st(K_2) \rightarrow st(K_3)$ must be a biadditive functor, i.e. an object of ${\HOM}(st(K_1), st(K_2);st(K_3))$.
Hence ${\HOM}({\bic}_0,{\bic}_0)$ is equivalent as Picard stack to ${\HOM}(st(K_1), st(K_2);st(K_3))$ via the following additive functor 
\begin{eqnarray}
\nonumber {\HOM}(st(K_1), st(K_2);st(K_3)) &\longrightarrow & {\HOM}({\bic}_0,{\bic}_0) \\
 \nonumber  F' & \mapsto & \big(b \mapsto b + IF'J (b) \big).
\end{eqnarray}
By (\ref{eq:tensorlinear}), ${\HOM}(st(K_1), st(K_2);st(K_3)) \cong {\HOM}(st(K_1) \otimes st(K_2),st(K_3))$ and so 
\begin{equation}
 [{\HOM}({\bic}_0,{\bic}_0)]= \tau_{\leq 0}{\R}{\Hom}\Big(\tau_{\geq -1}(K_1\otimes^{\LL}K_2),K_3 \Big),
\label{eq:proof(b)}
\end{equation}
 and in particular the group of isomorphism classes of additive functors from ${\bic}_0$ to itself is isomorphic to the group 
 $${\Hom}_{\cD(\bS)}(K_1\otimes^{\LL}K_2,K_3).$$
This implies that ${\Biext}^0(K_1,K_2;K_3) \cong {\Hom}_{\cD(\bS)}(K_1\otimes^{\LL}K_2,K_3).$ \\

In Section 6 we gives another proof of Theorem \ref{mainthm} \textbf{b}.
Remark that by (\ref{eq:proof(b)})
${\h}^{-1}([{\HOM}({\bic}_0,{\bic}_0)]) \cong {\Hom}_{\cD(\bS)}(K_1\otimes^{\LL}K_2,K_3[-1]).$
Since $K_i=[A_i \rightarrow B_i]$ are 1-motives, ${\Hom}(B_j,A_i)=0$ for $i,j=1,2,3$ (see \cite{Be09} Lemma 1.1.1), and hence the group ${\h}^{-1}([{\HOM}({\bic}_0,{\bic}_0)])$ is trivial.

%-------------------------------------------------------------------------
\section{The category $\Psi_{{\rL}^{..}}(G)$ and its homological interpretation}

Consider the following complex of 1-motives defined over $S$
\begin{equation}
\qquad \qquad R  \stackrel{D^R}{\longrightarrow}  Q \stackrel{D^Q}{\longrightarrow} P \longrightarrow 0
\label{eq:complex-1mot}
\end{equation} 
Explicitly, $R=[d^R:R^{-1} \rightarrow R^0], Q=[d^Q:Q^{-1} \rightarrow Q^0],P=[d^P:P^{-1} \rightarrow P^0]$ and  $D^R=(d^{R,-1},d^{R,0}), D^Q=(d^{Q,-1},d^{Q,0})$. This complex is a bicomplex ${\rL}^{..}$ of abelian sheaves on $\bS$,   
\[\xymatrix{
 R^{-1} \ar[d]_{d^R} \ar[r]^{d^{R, -1}} & Q^{-1} \ar[d]^{d^Q} \ar[r]^{d^{Q, -1}} & P^{-1} \ar[d]^{d^P} \ar[r] &0\\
R^0 \ar[r]^{d^{R, 0}} &  Q^0  \ar[r]^{d^{Q, 0}} & P^0  \ar[r] &0
}
\]  
where $P^0,P^{-1},Q^0,Q^{-1},R^0,R^{-1}$ are respectively in degrees $(0,0),(0,-1),(-1,0),\\(-1,-1),(-2,0),(-2,-1)$.
 Denote by ${\Tot}({\rL}^{..})$ its total complex. Let $G=[d^G:G^{-1} \rightarrow G^0]$ be a 1-motive defined over $S$.

\begin{defn}\label{psi}
Denote by $\Psi_{{\rL}^{..}}(G)$ the category 
\begin{enumerate}
       \item whose objects are pairs $(E,I)$ with $E$ an extension of 1-motives of $P$ by $G$ and $I$ a trivialization of the extension $(D^Q)^* E$ of $Q$ by $G$ obtained as pull-back of $E$ by $D^Q$.
        Moreover we require that the corresponding trivialization
           $(D^R)^* I$ of $(D^R)^* (D^Q)^* E$ is the trivialization arising from the isomorphism of transitivity $ (D^R)^* (D^Q)^* E \cong (D^Q \circ D^R)^*E$ and the relation $ D^Q \circ D^R = 0$. Note that to have such a trivialization $I$ is the same thing as to have a lifting $I: Q \rightarrow E$ of $D^Q:Q  \rightarrow P$ such that $ I \circ D^R = 0;$
      \item whose arrows $F: (E,I) \rightarrow (E',I')$ are morphisms of extensions $F: E \rightarrow E'$ of 1-motives compatible with the trivializations $I,I'$, i.e. we have an isomorphism of additive functors $F \circ I \cong I'$.
\end{enumerate}
 \end{defn}

In order to compute the homological interpretation of the category $\Psi_{{\rL}^{..}}(G)$, the language of Picard stacks will be very useful. Hence now we translate the construction of the category $\Psi_{{\rL}^{..}}(G)$ in terms of Picard stacks : Let 
$\ric=st(R), \qic=st(Q), \pic=st(P), \gic=st(G), D^\ric =st(D^R)$ and $D^\qic =st(D^Q)$. The complex of 1-motives (\ref{eq:complex-1mot}) furnishes the following complex of strictly commutative Picard $\bS$-stacks
\[\lic^.: \qquad \ric \stackrel{D^\ric}{\longrightarrow} \qic \stackrel{D^\qic}{\longrightarrow} \pic \stackrel{D^\pic}{\longrightarrow} 0\] 
with $\pic, \qic$ and $\ric$ in degrees 0,-1 and -2 respectively.
Via the equivalence of categories (\ref{st}), to the category $\Psi_{{\rL}^{..}}(G)$ is associated the category
$\Psi_{{\rL}^{..}}(G)$ $\Psi_{\lic^.}(\gic)$ 
\begin{enumerate}
       \item whose objects are pairs $(\eic,I)$ with $\eic$ an extension of $\pic$ by $\gic$ and $I$ a trivialization of the extension $(D^\qic)^* \eic$ of $\qic$ by $\gic$ obtained as pull-back of $\eic$ by $D^\qic$.
        Moreover we require that the corresponding trivialization
           $(D^\ric)^* I$ of $(D^\ric)^* (D^\qic)^* \eic$ is the trivialization arising from the isomorphism of transitivity $ (D^\ric)^* (D^\qic)^* \eic \cong  (D^\qic \circ D^\ric)^*\eic$ and the relation $ D^\qic \circ D^\ric \cong 0$. Note that to have such a trivialization $I$ is the same thing as to have a lifting $I: \qic \rightarrow \eic$ of $D^\qic:\qic  \rightarrow \pic$ such that $ I \circ D^\ric \cong 0;$
      \item whose arrows $F: (\eic,I) \rightarrow (\eic',I')$ are morphisms of extensions $F: \eic \rightarrow \eic'$ compatible with the trivializations $I,I'$, i.e. we have an isomorphism of additive functors $F \circ I \cong I'$.
\end{enumerate}

As observed at the end of section 2, extensions of 1-motives furnishes extensions of strictly commutative Picard stacks
and so the sum of extensions of strictly commutative Picard stacks introduced in \cite{Be10} 4.6 defines a group law on the set of isomorphism classes of objects of $\Psi_{{\rL}^{..}}(G)$. We denote this group by $\Psi_{{\rL}^{..}}^1(G).$
The neutral object of $\Psi_{{\rL}^{..}}(G)$ is the object $(E_0,I_0)$ where $E_0$ is the trivial extension $G \times P$ of $P$ by $G$ and $I_0$ is the trivialization $(Id_{Q},0)$
of the extension $(D^Q)^*E_0= G \times Q$ of $Q$ by $G$. We can consider $I_0$ as the lifting $(D^Q,0)$ of $D^Q: Q \rightarrow P$.

The monoid of automorphisms of an object $(E, I)$ of $\Psi_{{\rL}^{..}}(G)$ is canonically isomorphic to the monoid of automorphisms of $(E_0,I_0)$: to an automorphism $F :(E_0,I_0) \rightarrow (E_0,I_0)$ the canonical isomorphism associates the automorphism $F + Id_{(E,I)}$ of $(E_0,I_0)+ (E,I) \cong (E,I)$. 
The monoid of automorphisms of $(E_0,I_0)$ is a commutative group via the composition law $(F,G) \mapsto F+G$
(here $F+G$ is the automorphism of $(E_0,I_0) +(E_0,I_0) \cong (E_0,I_0)$). Hence we can conclude that the set of automorphisms of an object of $\Psi_{{\rL}^{..}}(G)$ is a commutative group that we denote by 
$\Psi_{{\rL}^{..}}^0(G).$

We can now state the homological interpretation of the groups $\Psi_{{\rL}^{..}}^i(G)$.

\begin{thm}\label{thm:psi-ext}
$$\Psi_{{\rL}^{..}}^i(G) \cong {\Ext}^i\big({\Tot}({\rL}^{..}),G\big)= {\Hom}_{\cD(\bS)}\big({\Tot}({\rL}^{..}),G[i]\big) \qquad \qquad i=0,1.$$
\end{thm}

\emph{Proof of the case i=0.} For this proof we will work with the category $\Psi_{\lic^.}(\gic)$. As observed above, $\Psi_{\lic^.}^0(\gic)$ is canonically isomorphic to the group of isomorphism classes of arrows from the neutral object  $(\eic_0,I_0)$ of $\Psi_{\lic^.}(\gic)$ to itself. By definition of arrows in the category $\Psi_{\lic^.}(\gic)$, the additive functor $F: \eic_0 \rightarrow \eic_0$ is an arrow from $(\eic_0,I_0)$ to itself if we have an isomorphism of additive functors  $F \circ D^\qic \cong 0$, i.e. if $F$ is an object of the strictly commutative Picard $\bS$-stack
\[\kic=\ker\big({\HOM}(\pic,\gic)  \stackrel{D^\qic}{\rightarrow} {\HOM}(\qic,\gic)\big).\]
Therefore we have the equality 
\begin{equation}
\Psi_{\lic^.}^0(\gic) = {\h}^0\big([\kic]\big)
\label{eq:kic}
\end{equation}
and in order to conclude, it is enough to compute the complex $[\kic]$ of $\cK^{[-1,0]}(\bS).$
By \cite{Be10} Lemma 3.4 we have  
\[ [\kic]= \tau_{\leq 0} \Big( MC\big(\tau_{\leq 0}{\R}{\Hom}([\pic],[\gic]) \stackrel{(d^{R, -1},d^{R, 0})}{\longrightarrow} 
\tau_{\leq 0}{\R}{\Hom}([\qic],[\gic]) \big)[-1] \Big).\]
Explicitly, we get 
\begin{equation}
[\kic]= \big[{\Hom}(P^0,G^{-1}) \stackrel{((d^G,d^P),d^{Q, 0})}{\longrightarrow} 
 K_1 +K_2\big]
\label{eq:psi0-1}
\end{equation}
 where 
\[
\begin{aligned}
\nonumber	K_1&= \ker\big({\Hom}(P^0,G^0) + {\Hom}(P^{-1},G^{-1}) \stackrel{(d^{Q, 0},d^{Q, -1})}{\rightarrow} {\Hom}(Q^0,G^0)+{\Hom}(Q^{-1},G^{-1}) \big)\\ 
\nonumber	K_2&=\ker \big({\Hom}(Q^0,G^{-1}) \stackrel{(d^G,d^Q)}{\rightarrow} {\Hom}(Q^0,G^0)+{\Hom}(Q^{-1},G^{-1}) \big).
\end{aligned}
\]
In order to simplify notation let $C^.: C^{-3} \rightarrow C^{-2} \rightarrow C^{-1} \rightarrow C^0$ be the total complex ${\Tot}([\lic^.])$. In particular $C^0=P^0, C^{-1}= P^{-1}+Q^0$ and $C^{-2}=Q^{-1}+R^0.$ 
The stupid filtration of the complexes $C^.$ and $G$ furnishes the spectral sequence
\begin{equation}
{\E}^{pq}_1= \bigoplus_{p_2-p_1=p} {\Ext}^q (C^{p_1},G^{p_2}) \Longrightarrow  {\Ext}^* (C^.,G).
\label{eq:suitesp}
\end{equation} 
This spectral sequence is concentrated in the region of the plane defined by $ -1 \leq p \leq 3$ and $q \geq 0$. We are interested on the total degrees -1 and 0. The rows $q=1$ and $q=0$ are
\small
\[
\begin{aligned}
& {\Ext}^1(C^{0},G^{-1}) \rightarrow  {\Ext}^1(C^{0},G^{0})  \oplus {\Ext}^1(C^{-1},G^{-1}) \rightarrow  {\Ext}^1(C^{-1},G^{0})  \oplus {\Ext}^1(C^{-2},G^{-1}) \rightarrow ...\\
&{\Hom}(C^{0},G^{-1})\stackrel{d_1^{-10}}{\rightarrow} {\Hom}(C^{0},G^{0})  \oplus {\Hom}(C^{-1},G^{-1})  \stackrel{d_1^{00}}{\rightarrow}  {\Hom}(C^{-1},G^{0})  \oplus {\Hom}(C^{-2},G^{-1}) \rightarrow ...
\end{aligned}
\]
\normalsize
Since ${\Ext}^1(C^{0},G^{-1})=0,$ i.e. the only extension of $[G^{-1} \rightarrow 0]$ by $[0 \rightarrow C^0]$ is the trivial one, we obtain  
\begin{eqnarray}
 \label{eq:2-cat}  {\Hom}_{{\cD}({\bS})}(C^.,G[-1]) &= & {\Ext}^{-1}(C^.,G) = {\E}^{-10}_2 = \ker ( d_1^{-10} ),\\
 \label{eq:D(S)=H(S)}  {\Hom}_{{\cD}({\bS})}(C^.,G) &= & {\Ext}^0(C^.,G) = {\E}^{00}_2 = \ker ( d_1^{00} ) / \im(d_1^{-10}) .
 \end{eqnarray}
 Comparing the above equalities with the explicit computation (\ref{eq:psi0-1}) of the complex $[\kic]$, we get 
 \begin{equation}
{\Ext}^i(C^.,G) = {\h}^i\big([\kic]\big) \qquad \qquad i=-1,0. 
\label{eq:ext-h(k)}
\end{equation}
 These equalities together with equality (\ref{eq:kic}) give the expected statement.

 \begin{rem}
 In the computation (\ref{eq:psi0-1}) the term ${\Hom}(P^{-1},G^0)$ does not appear because we work with the good truncation $\tau_{\leq 0}{\R}{\Hom}([\pic],[\gic])$. In the spectral sequence (\ref{eq:suitesp}) this term appear but we are interested in elements which become zero in ${\Hom}(P^{-1},G^0)$. 
 \end{rem}

\begin{rem} \label{rem:D(S)=H(S)} If ${\cH}(\bS)$ denotes the category of complexes of abelian sheaves on $\bS$ modulo 
homotopy, by equality (\ref{eq:D(S)=H(S)}) we have $ {\Hom}_{{\cD}({\bS})}(C^.,G)={\Hom}_{{\cH}({\bS})}(C^.,G).$
Moreover, since $P$ and $G$ are 1-motives we have that ${\Hom}(C^0,G^{-1})=0$ (\cite{Be09} Lemma 1.1.1) and so 
$ {\Hom}_{{\cD}({\bS})}(C^.,G)={\Hom}_{{\cK}({\bS})}(C^.,G).$
\end{rem}

\begin{rem}
 The category $\Psi_{\lic^.}(\gic)$ should be a 2-category, but it is just a category because we are working with strictly commutative Picard stacks defined by 1-motives. In fact, if $A$ is a group scheme which is locally for the \'etale
topology a constant group scheme defined by a finitely generated free
$\ZZ \,$-module and $B$ is an extension of an abelian scheme by a torus, then the group ${\Hom}(B,A)$ is trivial (\cite{Be09} Lemma 1.1.1). Because of (\ref{eq:kic}), (\ref{eq:2-cat}), (\ref{eq:ext-h(k)}), this implies that the group $\Psi_{\lic^.}^{-1}(\gic)$ of automorphisms of arrows from an object of $\Psi_{\lic^.}(\gic)$ to itself is trivial:
$$ \Psi_{\lic^.}^{-1}(\gic) \cong {\h}^{-1}\big([\kic]\big) \cong {\Ext}^{-1}(C^.,G)=\ker ( d_1^{-10} )=0.$$
 \end{rem}

\emph{Proof of the case i=1.} First we show how an object $(E,I)$ of $\Psi_{{\rL}^{..}}(G)$ defines a morphism ${\Tot}({\rL}^{..}) \rightarrow G[1]$ in the derived category $\cD(\bS)$. Recall that $E$ is an extension of 1-motives of $P$ by $G$. Denote $j:E \rightarrow P$ the surjective morphism underlying the extension $E$. Since the trivialization $I$ can be seen as a lifting $Q \rightarrow E$ of $D^Q:Q  \rightarrow P$ such that $ I \circ D^R = 0$, we have the following diagram in the category $\cK(\bS)$ of complexes of abelian sheaves on $\bS$ 
 \begin{equation}\label{eq:su}
\xymatrix{
{\rL}^{..}: & R  \ar[d] \ar[r]^{D^R} & Q \ar[d]_{i} \ar[r]^{D^Q} & P \ar[d]^{id_P} \ar[r]  & 0 \\
MC(j): &0 \ar[r] & E  \ar[r]^{j} & P  \ar[r] & 0
}
\end{equation}
where $ i \circ D^R =0$ and $ j \circ i= id_P \circ D^Q$. Putting the complex $P$ in degree 0, 
the above diagram gives an arrow  
\[ c(E,I): {\Tot}({\rL}^{..}) \longrightarrow MC(j)\]
in the derived category $\cD(\bS)$. The complex $E$ is an extension of 1-motives of $P$ by $G$ and so as observed at the end of section 2, $G$ is quasi-isomorphic to $\tau_{\leq 0}(MC(j)[-1])$. Hence we have constructed a canonical arrow
  \begin{eqnarray}\label{c}
 c: \Psi_{{\rL}^{..}}^1(G) & \longrightarrow & {\Hom}_{\cD(\bS)}\big({\Tot}({\rL}^{..}),G[1]\big)  \\
 \nonumber  (E,I) & \mapsto & c(E,I).
\end{eqnarray}
Now we will show that this arrow is bijective. The proof that this bijection is additive, i.e. that $c$ is an isomorphism of groups, is left to the reader.  \\

Injectivity:  Let $(E,I)$ be an object of $\Psi_{{\rL}^{..}}(G)$ such that the morphism $c(E,I)$ that it defines in $\cD(\bS)$ is the zero morphism.
The hypothesis that $c(E,I)$ is zero in $\cD(\bS)$ implies that there exists a resolution of $G$ 
$$V^0 \longrightarrow V^1 \longrightarrow V^2 \longrightarrow ...$$
 and a quasi isomorphism
\begin{equation}
\xymatrix{
  0  \ar[r] & E\ar[d]_{v^0} \ar[r]^{j} & P \ar[d]^{v^1} \ar[r]  & 0 &\\
 0  \ar[r] & V^0 \ar[r]^{k} & V^1 \ar[r] & V^2 \ar[r] & ...
}
\label{eq:qiso}
\end{equation}
such that the composite 
\[
  \xymatrix{
  R \ar[r]^{D^R} & Q\ar[d]_{i} \ar[r]^{D^Q} & P \ar[d]^{id_P} \ar[r]  & 0 &\\
  0  \ar[r] & E\ar[d]_{v^0} \ar[r]^{j} & P \ar[d]^{v^1} \ar[r]  & 0 &\\
 0  \ar[r] & V^0 \ar[r]^{k} & V^1 \ar[r] & V^2 \ar[r] & ...
}
\]
is homotopic to zero. We can assume $V^i \in \cK^{[-1,0]}(\bS)$ for all $i$ and $V^i=0$ for $i \geq 2$ (instead of the complex of complexes $(V^i)_i$ consider its good truncation in degree 1). The complex of complexes $(V^i)_i$ is a resolution of $G$, and so
the short sequence 
$$ 0 \longrightarrow G \longrightarrow V^0 \longrightarrow V^1 \longrightarrow 0$$
is exact, i.e. $V^0$ is an extension of $W$ by $G$.
Since the quasi-isomorphism (\ref{eq:qiso}) induces the identity on $G$, the extension $E$ is in fact the fibred product $P \times_{V^1} V^0$ of $P$ and $V^0$ over $V^1$. Therefore, the morphism $s: P \rightarrow V^0$ inducing the homotopy $(v^0, v^1) \circ c(\eic,I) \sim 0$, i.e. satisfying $ k\circ s= v^1 \circ id_{P},$ factorizes through a morphism
\[h: P \longrightarrow E = P \times_{V^1} V^0\]
satisfying 
\[ j \circ h =id_{P} \qquad \qquad h \circ D^Q = i .\]
These two equalities mean that $h$ splits the extension $E$, which is therefore the trivial extension of $P$ by $G$, and that $h$ is compatible with the trivializations $I$.
Hence we can conclude that the object $(E,I)$ lies in the isomorphism class of the zero object of $\Psi_{{\rL}^{..}}(G)$.\\

Surjectivity: Now we show that for any morphism $f$ of ${\Hom}_{\cD(\bS)}({\Tot}({\rL}^{..}),G[1])$, there is an element of $\Psi_{{\rL}^{..}}^1(G)$ whose image via $c$ is $f$.
The hypothesis that $f$ is an element of $\cD(\bS)$ implies that there exists a resolution of $G$ 
$$V^0 \longrightarrow V^1 \longrightarrow V^2 \longrightarrow ...$$
such that the morphism $f$ can be described in the category ${\cH}(\bS)$ of complexes modulo homotopy via the following diagram
	\begin{equation}\label{eq:su1}
  \xymatrix{
  R \ar[r]^{D^R} & Q\ar[d]_{v^0} \ar[r]^{D^Q} & P \ar[d]^{v^1} \ar[r]  & 0 &\\
 0  \ar[r] & V^0 \ar[r]^{k} & V^1 \ar[r] & V^2 \ar[r] &...
}
\end{equation}
We can assume $V^i \in \cK^{[-1,0]}(\bS)$ for all $i$ and $V^i=0$ for $i \geq 2$ (instead of the complex of complexes $(V^i)_i$ consider its good truncation in degree 1). Since the complex of complexes $(V^i)_i$ is a resolution of $G$,
the short sequence of complexes 
$$ 0 \longrightarrow G \longrightarrow V^0 \longrightarrow V^1 \longrightarrow 0$$
is exact, i.e. $V^0$ is an extension of $V^1$ by $G$.
 Consider the extension of $P$ by $G$
 \[Z=(v^1)^*V^0 = V^0 \times_{V^1} P\]
obtained as pull-back of $V^0$ via $v^{1}:P \rightarrow V^1.$ The pull-back of a short exact sequence is again a short exact sequence, and so $0 \rightarrow G \rightarrow Z \rightarrow P \rightarrow 0 $ is exact.
Moreover, as observed in Remark \ref{rem:ext-extcomplex}, since $P$ and $G$ are 1-motives the complex $Z$ is an extension of 1-motives. The condition $v^1 \circ D^Q = k \circ v^0 $ implies that $v^0 : Q\rightarrow V^0$ factories through a morphism
\[z:  Q \rightarrow Z\]
satisfying $l \circ z = D^Q$, with $l: Z \rightarrow P$ the canonical surjection of the extension $Z$. Moreover the equalities  $ v^0 \circ D^R =D^Q \circ D^R =0$ furnish $z \circ D^R=0$. 
Therefore the datum $(Z,z)$ is an object of the category  $\Psi_{{\rL}^{..}}(G)$. Consider now the morphism $c(Z,z): {\Tot}({\rL}^{..}) \rightarrow G[1]$ associated to
$(Z,z)$. By construction, the morphism $f$ (\ref{eq:su1})
is the composite of the morphism $c(Z,z)$ 
\[
\xymatrix{
 R  \ar[d] \ar[r]^{D^R} & Q \ar[d]_{z} \ar[r]^{D^Q} & P \ar[d]^{id_P} \ar[r]  & 0 \\
0 \ar[r] & Z  \ar[r]^{l} & P  \ar[r] & 0
}
\]
with the morphism 
\[\xymatrix{
  0  \ar[r] & Z\ar[d]_{h} \ar[r]^{l} & P \ar[d]^{v^1} \ar[r]  & 0 \\
 0  \ar[r] & V^0 \ar[r]^{k} & V^1 \ar[r] & 0  ,
} \]
where $h: Z=(v^1)^*V^0 \rightarrow V^0$ is the canonical projection underlying the pull-back $Z$. Since this last morphism is a morphism of resolutions of $G$ (inducing the identity on $G$), we can conclude that in the derived category $\cD(\bS)$ the morphism $f : {\Tot}({\rL}^{..}) \rightarrow G[1]$ (\ref{eq:su1}) is the morphism $c(Z,z)$.\\

Using the above homological description of the groups $\Psi_{{\rL}^{..}}^i(G)$ for $i=0,1$ we can study how the category $\Psi_{{\rL}^{..}}(G) $ varies 
with respect to the bicomplex ${\rL}^{..}$. Let $ R' \rightarrow Q' \rightarrow P' \rightarrow 0$ be another complex of 1-motives defined over $S$. Denote by ${{\rL}'}^{..}$ its total bicomplex. Consider a morphism of bicomplexes
\[ F: {{\rL}'}^{..} \longrightarrow {\rL}^{..} \]
given by the following commutative diagram 
\begin{equation}\label{variance:1}
\xymatrix{
 R' \ar[d]_{F^{-2}} \ar[r]^{D^{R'}} & Q'  \ar[d]_{F^{-1}} \ar[r]^{D^{Q'}} & P'  \ar[d]^{F^0} \ar[r]  & 0 \\
 R  \ar[r]_{D^R} & Q  \ar[r]_{D^Q} & P  \ar[r] & 0.
}
\end{equation}
The morphism $F$ defines a canonical functor
\[
F^*: \Psi_{{\rL}^{..}}(G) \longrightarrow \Psi_{{{\rL}'}^{..}}(G)
\]
as follows: if $(E,I)$ is an object of $\Psi_{{\rL}^{..}}(G)$, $F^*(E,I)$ is the object $(E',I')$ where
\begin{itemize}
  \item $E'$ is the extension $(F^0)^*E$ of $P'$ by $G$
obtained as pull-back of $E$ via $F^0: P' \rightarrow P$;
  \item $I'$ is the trivialization $(F^{-1})^*I $ of $(D^{Q'})^*E'$ induced by the trivialization $I$ of $(D^{Q})^*E$
via the commutativity of the first square of~(\ref{variance:1}).
\end{itemize}
The commutativity of the diagram~(\ref{variance:1}) implies that $(E',I')$ is in fact an object of $\Psi_{{{\rL}'}^{..}}(G)$ ( the condition $I' \circ D^{Q'} = 0$ is easily deducible from the corresponding conditions on $I$ and from the commutativity of the diagram~(\ref{variance:1})).

\begin{prop}\label{prop:psi-equiv}
Let $F: {{\rL}'}^{..} \rightarrow {\rL}^{..}$ be morphism of bicomplexes.
The corresponding functor $F^*: \Psi_{{\rL}^{..}}(G) \rightarrow \Psi_{{{\rL}'}^{..}}(G)$ is an equivalence of categories if and only if the homomorphisms 
$${\h}^i\big({\Tot}(F)\big):{\h}^i\big({\Tot}({{\rL}'}^{..})\big) \longrightarrow {\h}^i\big({\Tot}({\rL}^{..})\big)  \qquad i=0,1$$
 are isomorphisms.
\end{prop}

\begin{proof} The functor $F^*: \Psi_{{\rL}^{..}}(G) \rightarrow \Psi_{{{\rL}'}^{..}}(G)$ induces the homomorphisms 
\begin{equation}\label{1}
     \Psi_{{\rL}^{..}}^i(G) \longrightarrow \Psi_{{{\rL}'}^{..}}^i(G) \qquad i=0,1.
\end{equation}
On the other hand the morphism of bicomplexes  $F: {{\rL}'}^{..} \rightarrow {\rL}^{..}$ defines the homomorphisms
\begin{equation}\label{2}
    {\Ext}^i\big({\Tot}({\rL}^{..}), -\big) \longrightarrow {\Ext}^i\big({\Tot}({{\rL}'}^{..}), -\big) \qquad i \in {\ZZ}.
\end{equation}
Since the homomorphisms~(\ref{1}) and~(\ref{2}) are compatible with the canonical isomorphisms obtained in Theorem~\ref{thm:psi-ext},
 the following diagrams (with $i=0,1$) are commutative:
\[
\begin{array}{ccc}
 \Psi_{{\rL}^{..}}^i(G)&\rightarrow &{\Ext}^i\big({\Tot}({\rL}^{..}),G\big)\\
 \downarrow &  & \downarrow \\
 \Psi_{{{\rL}'}^{..}}^i(G)& \rightarrow & {\Ext}^i\big({\Tot}({{\rL}'}^{..}),G\big).
\end{array}
\]
The functor $F^*: \Psi_{{\rL}^{..}}(G) \rightarrow \Psi_{{{\rL}'}^{..}}(G)$ is an equivalence of categories if and only if
 the homomorphisms~(\ref{1}) are isomorphisms, and so using the above commutative diagrams 
we are reduced to prove that the homomorphisms~(\ref{2}) are isomorphisms if and only if the homomorphisms
 ${\h}^i\big({\Tot}(F)\big):{\h}^i\big({\Tot}({{\rL}'}^{..})\big) \rightarrow {\h}^i\big({\Tot}({\rL}^{..})\big)$ 
are isomorphisms. This last assertion is clearly true.
\end{proof}

%-------------------------------------------------------------------------
\section{Geometrical description of $\Psi_{{\rL}..}(G)$}

In this section we switch from cohomological notation to homological.

Let $K=[u:A \rightarrow B]$ be a 1-motive defined over $S$ with $A$ in degree 1 and $B$ in degree 0. We start constructing \textbf{a canonical flat partial resolution} ${\rL}..(K)$ of the complex $K$. But before, we introduce the following notations: if $P$ is an abelian sheaf on $\bS$, we denote by $[p]$ the point of ${\ZZ}[P](U)$ defined by the point $p$ of $P(U)$ with $U$ an $S$-scheme. In an analogous way, if $p,q$ and $r$ are points of $P(U)$ we denote by $[p,q]$, $[p,q,r]$ the elements of ${\ZZ}[P\times P](U)$ and ${\ZZ}[P\times P \times P](U)$ respectively.

Consider the following complexes of ${\cD}^{[1,0]}(\bS)$
\begin{eqnarray}
\nonumber P &=& [{\ZZ}[A] \stackrel{D_{00}}{\longrightarrow}  {\ZZ}[B]]\\
\label{eq:PQR} Q &=& [0 \longrightarrow {\ZZ}[B \times B] ]\\
\nonumber R &=& [0 \longrightarrow {\ZZ}[B \times B] + {\ZZ}[B \times B \times B]]
\end{eqnarray}
and the following morphisms of complexes
\begin{eqnarray}
\nonumber (\epsilon_1,\epsilon_0) &:& P\longrightarrow K\\
\nonumber (0,d_{00}) &:& Q\longrightarrow P\\
\nonumber (0,d_{01}) &:& R\longrightarrow Q
\end{eqnarray}
where for any $U$ and for any $a\in A(U), b_1, b_2,b_3 \in B(U),$ we set 
\begin{eqnarray}
\nonumber \epsilon_0[b] &=& b \\
\nonumber \epsilon_1[a] &=& a \\
\nonumber d_{00}[b_1,b_2] &=& [b_1+b_2] -[b_1]-[b_2]\\
\label{Dd} d_{01}[b_1,b_2] &=& [b_1,b_2] -[b_2,b_1]\\
\nonumber d_{01}[b_1,b_2,b_3] &=& [b_1+b_2,b_3] -[b_1,b_2+b_3]+[b_1,b_2]-[b_2,b_3]\\
\nonumber D_{00}[a] &=& [u(a)].
\end{eqnarray}
These data define the bicomplex ${\rL}..(K)$
\[
\begin{array}{ccccccccc}
 && \scriptstyle{{\rL}_{2*}(K)} & & \scriptstyle{{\rL}_{1*}(K)} & & \scriptstyle{{\rL}_{0*}(K)} & &  \\
  && \overbrace{}&  & \overbrace{} &  & \overbrace{} &  &  \\
\scriptstyle{ {\rL}_{*3}(K)}  & \big\{ & &  & 0 &  & 0&  &  \\
   && &  & \downarrow &  & \downarrow &  &  \\
\scriptstyle{R={\rL}_{*2}(K)} & \big\{ & 0 & \rightarrow & 0 & \rightarrow & {\ZZ}[B \times B]+ {\ZZ}[B \times B \times B] & \rightarrow & 0 \\
  && &  & \downarrow &  & \downarrow \scriptstyle{d_{01}}&  &  \\
\scriptstyle{Q={\rL}_{*1}(K)} & \big\{ & 0 & \rightarrow & 0 & \rightarrow & {\ZZ}[B \times B] & \rightarrow & 0 \\
  && &  & \downarrow &  & \downarrow \scriptstyle{d_{00}}&  &  \\
\scriptstyle{P={\rL}_{*0}(K)} & \big\{ & 0 & \rightarrow & {\ZZ}[A] & \stackrel{D_{00}}{\rightarrow} & {\ZZ}[B] & \rightarrow & 0 \\
  && &  & \downarrow \scriptstyle{\epsilon_1} &  & \downarrow \scriptstyle{\epsilon_0} &  &  \\
\scriptstyle{K} & \big\{ & 0 & \rightarrow & A & \stackrel{u}{\rightarrow} & B & \rightarrow & 0
\end{array}
\]
which satisfies ${\rL}_{ij}(K)= 0$ for $(ij) \neq (00),(01),(02),(10)$ and which is endowed with an augmentation map 
$\epsilon . =(\epsilon_1,\epsilon_0) : P\rightarrow K$.
 Note that the relation $\epsilon_0 \circ d_{00} =0$ is just the group law on $B$, and the relation $d_{00}\circ d_{01} =0$ decomposes in two relations which express the commutativity and the associativity of the group law on $B$.
This augmented bicomplex ${\rL}..(K)$ depends functorially on $K$: in fact, any morphism $f:K \rightarrow K'$ of 1-motives furnishes a commutative diagram
\[\begin{array}{ccc}
  {\rL}..(K) & \stackrel{{\rL}..(f)}{\longrightarrow} & {\rL}..(K') \\
 \epsilon . \downarrow &  & \downarrow \epsilon. \\
  K & \stackrel{f}{\longrightarrow} & K'.
\end{array}\]
Moreover the components of the bicomplex ${\rL}..(K)$ are flat since they are free $\ZZ$-modules. In order to conclude that ${\rL}..(K)$ is a canonical flat partial resolution of $K$ we need the following Lemma. Let $K'=[u':A' \rightarrow B']$ be a 1-motive defined over $S$.

\begin{lem}\label{lem:ext-geom}
The category ${\bExt}(K,K')$ of extensions of $K$ by $K'$ is equivalent to the category $\Psi_{{\rL}..(K)}(K'):$
  \begin{equation}
    {\bExt}(K;K') \simeq \Psi_{{\rL}..(K)}(K').
  \end{equation}
\end{lem}

\begin{proof} In order to describe explicitly the objects of the category $\Psi_{{\rL}..(K)}(K')$ we use the description~(\ref{1.4}) of extensions of free commutative groups in terms of torsors:
\begin{itemize}
  \item an extension of ${\ZZ}[B]$ by $B'$ is a $(B')_{B}$-torsor,
  \item an extension of ${\ZZ}[A]$ by $B'$ is a $(B')_{A}$-torsor,
  \item  an extension of
$ {\ZZ}[ B \times B] $ by $B'$ is a $(B')_{B \times B}$-torsor, and finally
  \item an extension of ${\ZZ}[B \times B]+ {\ZZ}[B \times B \times B] $ by $B'$ consists of a couple of a $(B')_{B \times B}$-torsor
      and a $(B')_{B \times B \times B}$-torsor.
\end{itemize}
According to these considerations an object $(E,I)$ of $\Psi_{{\rL}..(K)}(K')$ consists of
\begin{enumerate}
  \item an extension $E$ of $P=[D_{00}:{\ZZ}[A] \rightarrow  {\ZZ}[B]]$ 
by $K'=[u':A' \rightarrow B']$, i.e.
	\begin{enumerate}
	\item a $B'$-torsor $E$ over $B$,
	\item a trivializations $\beta$ of the $B'$-torsor
      $D_{00}^*E$ over $A$ obtained as pull-back of $E$ via $D_{00}: {\ZZ}[A] \rightarrow {\ZZ}[B],$
      \item an homomorphism $\gamma: A \rightarrow A'$ such that the composite $u' \circ \gamma $
   is compatible with $\beta;$ 
  	\end{enumerate}
  \item a trivialization $I$ of the extension $(0,d_{00})^*E$ of $Q$ by $K'$  obtained as pull-back of $E$ by $(0,d_{00}): Q \rightarrow P$, i.e. a trivialization $I$ of the $B'$-torsor $d_{00}^*E$ over $B \times B$ obtained as pull-back of $E$ via $ d_{00}: {\ZZ}[B \times B] \rightarrow {\ZZ}[B]$. This trivialization can be interpreted as a group law on the fibres of the $B'$-torsor $E$:
      \[
      +: E_{b_1} ~ E_{b_2} \longrightarrow E_{b_1+b_2}
      \]
where $b_1,b_2$ are points of $B(U)$ with $U$ an $S$-scheme.
The compatibility of $I$ with the relation $(0,d_{00}) \circ (0,d_{01}) =0$ imposes on the datum $(E,+)$ two relations through the two torsors over $B \times B$ and $B \times B \times B$. These two relations are the relations of commutativity and of associativity of the group law $+$, which mean that $+$ defines over $E$ a structure of commutative extension of $B$ by $B'$.
  \end{enumerate}
Hence the object $(E,+,\beta,\gamma)$ of $\Psi_{{\rL}..(K)}(K')$ is an extension of $K$ by $K'$ and we can conclude that the category $\Psi_{{\rL}..(K)}(K')$ is equivalent to the category ${\bExt}(K,K')$. 
\end{proof}

\begin{prop}\label{resolpart}
The augmentation map $\epsilon . : {\rL .}_{0}(K) \rightarrow K$ induces the isomorphisms ${\h}_1({\Tot}( {\rL}..(K))) \cong {\h}_1(K)$ and
 ${\h}_0({\Tot}( {\rL}..(K))) \cong {\h}_0(K).$
\end{prop}

\begin{proof} Applying Proposition~\ref{prop:psi-equiv} to
 the augmentation map $\epsilon . : {\rL .}_{0}(K) \rightarrow K$,
 we just have to prove that for any 1-motive $K'=[u': A' \rightarrow B']$ 
 the functor
 $$\epsilon .^* : \Psi_{K}(K') \rightarrow
 \Psi_{{\rL}..(K)}(K')$$
 is an equivalence of categories (in the symbol $\Psi_{K}(K')$, $K$ is seen as a bicomplex whose only non trivial entries are $A$ in degree (10) and $B$ in degree (00)). According to definition~\ref{psi}, it is clear that the category
 $\Psi_{K}(K')$ is equivalent to the category ${\bExt}(K,K')$
 of extensions of $K$ by $K'$. On the other hand, by Lemma~\ref{lem:ext-geom} also the category
 $\Psi_{{\rL}..(K)}(K')$ is equivalent to the category ${\bExt}(K,K')$. Hence we can conclude.
\end{proof}

 Let $K_i=[u_i: A_i \rightarrow B_i]$ (for $i=1,2,3$) be 1-motives defined over $S$ and let  ${\rL}..(K_i)$ be its canonical flat partial resolution. Denote by
 ${\rL}..(K_1, K_2)$ the bicomplex ${\rL}..(K_1) \otimes {\rL}..(K_2).$ 

\begin{thm}\label{thm:psi-geom}
The category ${\bBiext}(K_1,K_2;K_3)$ of biextensions of $(K_1,K_2)$ by $K_3$ is equivalent to the category $\Psi_{\tau_{\leq (1*)}{\rL}..(K_1,K_2)}(K_3):$
\[
    {\bBiext}(K_1,K_2;K_3) \simeq \Psi_{\tau_{\leq (1*)}{\rL}..(K_1,K_2)}(K_3)
\]
\end{thm}

\begin{proof}
Explicitly, the non trivial components of ${\rL}_{ij}(K_1,K_2)$ are
\begin{eqnarray}
\nonumber {\rL}_{00}(K_1,K_2) &=& {\rL}_{00}(K_1) \otimes {\rL}_{00}(K_2)\\
\nonumber  &=& {\ZZ}[B_1 \times B_2] \\
\nonumber {\rL}_{01}(K_1,K_2) &=& {\rL}_{00}(K_1) \otimes {\rL}_{01}(K_2)+{\rL}_{01}(K_1) \otimes {\rL}_{00}(K_2)\\
\nonumber  &=&{\ZZ}[B_1 \times B_2 \times B_2]+ {\ZZ}[B_1 \times B_1\times B_2] \\
\nonumber {\rL}_{02}(K_1,K_2) &=&{\rL}_{00}(K_1) \otimes {\rL}_{02}(K_2)+{\rL}_{02}(K_1) \otimes {\rL}_{00}(K_2)+{\rL}_{01}(K_1) \otimes {\rL}_{01}(K_2)\\
 \nonumber  &=&{\ZZ}[B_1 \times B_2 \times B_2]+{\ZZ}[B_1 \times B_2 \times B_2 \times B_2]+ \\
\nonumber  &=& {\ZZ}[B_1 \times B_1 \times B_2]+{\ZZ}[B_1 \times B_1 \times B_1 \times B_2]+\\
\nonumber  &=& {\ZZ}[B_1 \times B_1 \times B_2 \times B_2]\\
\nonumber {\rL}_{03}(K_1,K_2) &=&{\rL}_{01}(K_1) \otimes {\rL}_{02}(K_2)+{\rL}_{02}(K_1) \otimes {\rL}_{01}(K_2)\\
\nonumber {\rL}_{04}(K_1,K_2) &=&{\rL}_{02}(K_1) \otimes {\rL}_{02}(K_2)\\
\nonumber {\rL}_{10}(K_1,K_2) &=& {\rL}_{10}(K_1) \otimes {\rL}_{00}(K_2)+{\rL}_{00}(K_1) \otimes {\rL}_{10}(K_2)\\
\nonumber  &=&{\ZZ}[A_1 \times B_2 ]+ {\ZZ}[B_1 \times A_2]\\
\nonumber {\rL}_{11}(K_1,K_2) &=& {\rL}_{10}(K_1) \otimes {\rL}_{01}(K_2)+{\rL}_{01}(K_1) \otimes {\rL}_{10}(K_2)\\
\nonumber  &=&{\ZZ}[A_1 \times B_2 \times B_2]+ {\ZZ}[B_1 \times B_1\times A_2]\\
\nonumber {\rL}_{12}(K_1,K_2) &=& {\rL}_{10}(K_1) \otimes {\rL}_{02}(K_2)+{\rL}_{02}(K_1) \otimes {\rL}_{10}(K_2)\\
\nonumber {\rL}_{20}(K_1,K_2) &=& {\rL}_{10}(K_1) \otimes {\rL}_{10}(K_2)\\
\nonumber  &=&{\ZZ}[A_1 \times A_2]
\end{eqnarray}
The differential operators of ${\rL}..(K_1, K_2)$ can be computed from the below diagram, where we don't have written the identity homomorphisms in order to avoid too heavy notation (for example instead of
$(id \times D_{00}^{K_2},D_{00}^{K_1} \times id)$ we have written just $(D_{00}^{K_2},D_{00}^{K_1})$):
\begin{equation}\label{diffop}
\xymatrix{
& \scriptstyle{{\rL}_{2*}(K)}  & \scriptstyle{{\rL}_{1*}(K)}  & \scriptstyle{{\rL}_{0*}(K)}  \\
\scriptstyle{ {\rL}_{*2}(K)} &&  0 \ar[r] \ar[d]  &  {\rL}_{02}(K_1,K_2)\ar[d]^{d_{01}^{K_2}+ d_{01}^{K_1} + (d_{00}^{K_1}, d_{00}^{K_2})}  \\
 \scriptstyle{ {\rL}_{*1}(K)}&&  {\rL}_{11}(K_1,K_2)\qquad \ar[r]^{D_{00}^{K_1}+D_{00}^{K_2} }
  \ar[d]_{d_{00}^{K_2}+ d_{00}^{K_1}}& \qquad {\rL}_{01}(K_1,K_2)\ar[d]^{d_{00}^{K_2}+ d_{00}^{K_1}} \\
\scriptstyle{ {\rL}_{*0}(K)} &  {\rL}_{20}(K_1,K_2)\qquad \ar[r]^{(D_{00}^{K_2},D_{00}^{K_1})} &  \qquad {\rL}_{10}(K_1,K_2) \qquad \ar[r]^{D_{00}^{K_1}+D_{00}^{K_2}}   & \qquad {\rL}_{00}(K_1,K_2).
}
\end{equation}
These operators have to satisfy the well-known conditions on differential operators of bicomplexes that we recall explicitly here:
\begin{itemize}
  \item the following sequences are exact:
  \begin{equation}\label{exact1}
   {\ZZ}[B_1 \times B_2 \times B_2]+{\ZZ}[B_1 \times B_2 \times B_2 \times B_2]
   \stackrel{d_{01}^{K_2} }{\longrightarrow}
   {\ZZ}[B_1 \times B_2 \times B_2]
   \stackrel{d_{00}^{K_2} }{\longrightarrow}
   {\ZZ}[B_1 \times B_2]
  \end{equation}
  \begin{equation}\label{exact2}
    {\ZZ}[B_1 \times B_1 \times B_2]+{\ZZ}[B_1 \times B_1 \times B_1 \times B_2]
   \stackrel{d_{01}^{K_1} }{\longrightarrow}
   {\ZZ}[B_1 \times B_1 \times B_2]
   \stackrel{d_{00}^{K_1} }{\longrightarrow}
   {\ZZ}[B_1 \times B_2]
  \end{equation}
   \item the following diagrams are anticommutative:
 \begin{equation}\label{anti1}
      \begin{array}{ccc}
        {\ZZ}[B_1 \times B_1 \times B_2 \times B_2] & \stackrel{d_{00}^{K_2} }{\longrightarrow}  & {\ZZ}[B_1 \times B_1 \times B_2]  \\
 \scriptstyle{d_{00}^{K_1}}    \downarrow & & \downarrow \scriptstyle{d_{00}^{K_1}}\\
        {\ZZ}[B_1 \times B_2 \times B_2]  &  \stackrel{d_{00}^{K_2} }{\longrightarrow}  & {\ZZ}[B_1 \times B_2] \\
      \end{array}
 \end{equation}
 \begin{equation}\label{anti2}
   \begin{array}{ccc}
        {\ZZ}[A_1  \times B_2 \times B_2] & \stackrel{D_{00}^{K_1} }{\longrightarrow}  & {\ZZ}[B_1 \times B_2 \times B_2]  \\
   \scriptstyle{d_{00}^{K_2}}    \downarrow &  & \downarrow \scriptstyle{d_{00}^{K_2}}\\
        {\ZZ}[A_1 \times B_2 ]  &  \stackrel{D_{00}^{K_1} }{\longrightarrow}  & {\ZZ}[B_1 \times B_2] \\
      \end{array}
 \end{equation}
 \begin{equation}\label{anti3}
     \begin{array}{ccc}
        {\ZZ}[B_1  \times B_1 \times A_2] & \stackrel{D_{00}^{K_2} }{\longrightarrow}  & {\ZZ}[B_1 \times B_1 \times B_2]  \\
   \scriptstyle{ d_{00}^{K_1}}    \downarrow &  & \downarrow \scriptstyle{ d_{00}^{K_1}}\\
        {\ZZ}[B_1 \times A_2 ]  &  \stackrel{D_{00}^{K_2} }{\longrightarrow}  & {\ZZ}[B_1 \times B_2] \\
      \end{array}
 \end{equation}
 \begin{equation}
     \begin{array}{ccc}\label{anti4}
        {\ZZ}[A_1 \times A_2] & \stackrel{D_{00}^{K_2} }{\longrightarrow}  & {\ZZ}[A_1 \times B_2]  \\
  \scriptstyle{ D_{00}^{K_1} }   \downarrow &  & \downarrow \scriptstyle{ D_{00}^{K_1}} \\
        {\ZZ}[B_1 \times A_2 ]  &  \stackrel{D_{00}^{K_2} }{\longrightarrow}  & {\ZZ}[B_1 \times B_2] \\
      \end{array}
 \end{equation}
\end{itemize}
The bicomplex $\tau_{\leq (1*)}{\rL}..(K_1,K_2)$ is furnished by the bicomplex (\ref{diffop}) where instead of ${\rL}_{10}(K_1)$ we have 
\begin{eqnarray}
\nonumber {\rL}_{10}'(K_1,K_2) &=& {\rL}_{10}(K_1,K_2) \Big/ (D_{00}^{K_2},D_{00}^{K_1}) \; {\rL}_{20}(K_1,K_2)\\
\label{eq:troncato} &=& {\ZZ}[A_1 \times B_2] + {\ZZ}[B_1 \times A_2] \Big/ (id \times u_2)+(u_1 \times id) \;{\ZZ}[A_1 \times A_2]
\end{eqnarray}
In order to describe explicitly the objects of $\Psi_{\tau_{\leq (1*)}{\rL}..(K_1,K_2)}(K_3)$ we use the description~(\ref{1.4}) of extensions of free commutative groups in terms of torsors:
\begin{itemize}
  \item an extension of ${\rL}_{00}(K_1,K_2)$ by $B_3$ is a $(B_3)_{B_1 \times B_2}$-torsor,
  \item  an extension of
${\rL}_{10}'(K_1,K_2)$ by $B_3$ consists of a $(B_3)_{A_1 \times B_2}$-torsor and a $(B_3)_{B_1 \times A_2}$-torsor,
  \item an extension of ${\rL}_{02}(K_1,K_2)$ by $B_3$ consists of a system of 5 torsors under the groups deduced from $B_3$ by base change over the bases $B_1 \times B_2 \times B_2,~B_1 \times B_2 \times B_2 \times B_2,~ B_1 \times B_1 \times B_2,~
B_1 \times B_1 \times B_1 \times B_2,~B_1 \times B_1 \times B_2 \times B_2.$
\end{itemize}
By these considerations an object $(E,I)$ of $\Psi_{\tau_{\leq (1*)}{\rL}..(K_1,K_2)}(K_3)$ consists of
\begin{enumerate}
  \item an extension $E$ of $[D_{00}^{K_1}+D_{00}^{K_2}: {\rL}_{10}'(K_1,K_2) \rightarrow {\rL}_{00}(K_1,K_2)] $ by $K_3$, i.e. 
  \begin{enumerate}
       \item a $B_3$-torsor $E$ over $B_1 \times B_2$,
        \item a couple of trivializations $(\Psi_1,\Psi_2)$ of the couple of $B_3$-torsors \\
      $((D_{00}^{K_1}\times id )^*E,(id \times D_{00}^{K_2})^*E )$ over $A_1 \times B_2$ and $B_1 \times A_2 $ respectively, which are the pull-back of $E$ via 
      $$(D_{00}^{K_1}\times id )+(id \times D_{00}^{K_2}): {\ZZ}[A_1 \times B_2]+{\ZZ}[B_1 \times A_2] \rightarrow {\ZZ}[B_1 \times B_2];$$
       We consider the factor ${\rL}_{10}'(K_1,K_2)$ (\ref{eq:troncato}) instead of $ {\rL}_{10}(K_1,K_2)$ and this means that the restriction of the trivializations $(\Psi_1,\Psi_2)$ have to coincide over $A_1 \times A_2$,
       \item a homomorphism $\gamma: {\ZZ}[A_1]\otimes {\ZZ}[A_2] \rightarrow A_3 $  such that the composite ${\ZZ}[A_1]\otimes {\ZZ}[A_2] \stackrel{\gamma}{\rightarrow} {\ZZ}[A_1]\otimes {\ZZ}[A_2] \stackrel{u_3}{\rightarrow} B_3$ is compatible with the restriction of the trivializations $\Psi_1,\Psi_2$ over ${\ZZ}[A_1]\otimes {\ZZ}[A_2]$.
  \end{enumerate}
  \item a trivialization $I$ of the extension $(d_{00}^{K_2}+ d_{00}^{K_1},d_{00}^{K_2}+ d_{00}^{K_1} )^*E $ of 
  $[D_{00}^{K_1}+D_{00}^{K_2}:{\rL}_{11}(K_1,K_2) \rightarrow {\rL}_{01}(K_1,K_2)]$ by $K_3$ obtained as pull-back of $E$ via 
$$(d_{00}^{K_2}+ d_{00}^{K_1},d_{00}^{K_2}+ d_{00}^{K_1} ): [{\rL}_{11}(K_1,K_2) \rightarrow {\rL}_{01}(K_1,K_2)] \longrightarrow  [{\rL}_{10}'(K_1,K_2) \rightarrow {\rL}_{00}(K_1,K_2)],$$
 i.e. a couple of trivializations $\alpha=(\alpha_1,\alpha_2)$ of the couple of $B_3$-torsors over $B_1 \times B_2 \times B_2$ and $B_1 \times B_1 \times B_2$ which are the pull-back of $E$ via $(id \times d_{00}^{K_2})+(d_{00}^{K_1} \times id) : {\ZZ}[B_1 \times B_2 \times B_2]+{\ZZ}[B_1 \times B_1 \times B_2]
      \rightarrow {\ZZ}[B_1 \times B_2]$. The trivializations $(\alpha_1,\alpha_2)$ can be viewed as two group laws on the fibres of the $B_3$-torsor $E$ over $B_1 \times B_2$:
      \[
      +_2: E_{b_1,b_2} ~ E_{b_1,b'_2} \longrightarrow E_{b_1,b_2+b'_2} \qquad \qquad  +_1: E_{b_1,b_2} ~ E_{b'_1,b_2} \longrightarrow E_{b_1+b'_1,b_2}
      \]
      where $b_2,b'_2$ (resp. $b_1,b'_1$) are points of $B_2(U)$ (resp. of $B_1(U)$) with $U$ any $S$-scheme.\\
 The trivialization $I$, i.e. the two group laws, must be compatible with the trivializations $(\Psi_1,\Psi_2)$ underlying the trivialization $E$. This compatibility is expressed through the 2 torsors arising from the factors ${\rL}_{11}(K_1,K_2)$:
          \begin{itemize}
          \item the anticommutative diagram~(\ref{anti2}) furnishes a relation of compatibility between the group law $+_2$ of $E$ and the trivialization $\Psi_1$ of the pull-back $(D_{00}^{K_1}\times id)^*E$ of $E$ over $A_1 \times B_2$, which means that $\Psi_1$ is an additive section;
          \item the anticommutative diagram~(\ref{anti3}) furnishes a relation of compatibility between the group law $+_1$ of $E$ and the trivialization $\Psi_2$ of the pull-back $(id \times D_{00}^{K_2})^*E$ of $E$ over $B_1 \times A_2$, which means that also $\Psi_2$ is an additive section.
          \end{itemize}
Finally, the compatibility of $I$ with the relation 
 $$\big(d_{00}^{K_2}+ d_{00}^{K_1},d_{00}^{K_2}+ d_{00}^{K_1} \big) \circ \big(d_{01}^{K_2}+ d_{01}^{K_1} + (d_{00}^{K_1}, d_{00}^{K_2})\big) =0$$
 imposes on the datum $(E,+_1,+_2)$ 5 relations of compatibility
       through the system of 5 torsors over $B_1 \times B_2 \times B_2,~B_1 \times B_2 \times B_2 \times B_2,~ B_1 \times B_1 \times B_2,~B_1 \times B_1 \times B_1 \times B_2,~B_1 \times B_1 \times B_2 \times B_2$ arising from ${\rL}_{02}(K_1,K_2):$ 
   \begin{itemize}
         \item the exact sequence~(\ref{exact1}) furnishes the two relations of commutativity and of associativity of the group law $+_2$, which mean that $+_2$ defines over $E$ a structure of commutative extension of $(B_2)_{B_1}$ by $(B_3)_{B_1}$;
          \item the exact sequence~(\ref{exact2}) expresses the two relations of commutativity and of associativity of the group law $+_1$, which mean that $+_1$ defines over $E$ a structure of commutative extension of $(B_1)_{B_2}$ by $(B_3)_{B_2}$;
          \item the anticommutative diagram~(\ref{anti1}) means that these two group laws are compatible.
          \end{itemize}
   Therefore these 5 conditions implies that the $B_3$-torsor $E$ is endowed with a structure of biextension of $(B_1,B_2)$ by $B_3$.
  \end{enumerate}
The object $(E,\Psi_1,\Psi_2,\gamma)$ of $\Psi_{\tau_{\leq (1*)}{\rL}..(K_1,K_2)}(K_3)$ is therefore a biextension of $(K_1,K_2)$ by $K_3$.
\end{proof}

In the above proof we have not used diagram~(\ref{anti4}) because we work with the truncated bicomplex $\tau_{\leq (1*)}{\rL}..(K_1,K_2)$ (see (\ref{eq:troncato})).

% ------------------------------------------------------------------------
\section{Proof of theorem 0.1 (\textbf{a})}
Let $K_i=[A_i \stackrel{u_i}{\rightarrow} B_i]$ (for $i=1,2,3$) be three 1-motives defined over $S$.
Denote by ${\rL}..(K_i)$ (for $i=1,2$) the canonical flat partial resolution of $K_i$ introduced in \S 5. According to Proposition~\ref{resolpart}, there exists an arbitrary flat resolution ${\rL}'..(K_i)$ (for $i=1,2$) of $K_i$ such that the groups ${\Tot}( {\rL}..(K_i))_j$ and ${\Tot}( {\rL}'..(K_i))_j$ are isomorphic for $j=0,1$. We have therefore two canonical homomorphisms of bicomplexes
$$ {\rL}..(K_1) \longrightarrow {\rL}'..(K_1)  \qquad
{\rL}..(K_2) \longrightarrow {\rL}'..(K_2)$$
inducing a canonical homomorphism between the corresponding total complexes
$${\Tot}({\rL}..(K_1) \otimes {\rL}..(K_2)) \longrightarrow {\Tot} ({\rL}'..(K_1) \otimes {\rL}'..(K_2)) $$
which is an isomorphism in degrees 0 and 1.
Denote by
 ${\rL}..(K_1, K_2)$ (resp. ${\rL}'..(K_1, K_2)$) the bicomplex ${\rL}..(K_1) \otimes {\rL}..(K_2)$
 (resp. ${\rL}'..(K_1) \otimes {\rL}'..(K_2)$).
 Remark that ${\Tot}({\rL}'..(K_1,K_2))$ represents $K_1 {\buildrel {\scriptscriptstyle \LL} \over \otimes} K_2$ in the derived category $\cD(\bS)$:
 \[{\Tot}({\rL}'..(K_1,K_2)) =  K_1 {\buildrel {\scriptscriptstyle \LL} \over \otimes} K_2.\]
By Proposition~\ref{prop:psi-equiv} we have the equivalence of categories
$$\Psi_{\tau_{\leq (1*)}{\rL}..(K_1,K_2)}(K_3) \simeq \Psi_{\tau_{\leq (1*)}{\rL}'..(K_1,K_2)}(K_3).$$
Hence applying Theorem~\ref{thm:psi-geom}, which furnishes the following geometrical description of the category $\Psi_{\tau_{\leq (1*)}{\rL}..(K_1,K_2)}(K_3)$:
$${\bBiext}(K_1,K_2;K_3) \simeq \Psi_{\tau_{\leq (1*)}{\rL}..(K_1,K_2)}(K_3), $$
and applying Theorem~\ref{thm:psi-ext}, which furnishes the following homological description of the groups $\Psi_{\tau_{\leq (1*)}{\rL}'..(K_1,K_2)}^i(K_3)$ for $i=0,1$:
$$\Psi_{\tau_{\leq (1*)}{\rL}'..(K_1,K_2)}^i(K_3)\cong  {\Ext}^i\big({\Tot}\big(\tau_{\leq (1*)}{\rL}'..(K_1,K_2)\big),K_3\big)\cong {\Ext}^i(K_1 {\buildrel {\scriptscriptstyle \LL} \over \otimes} K_2,K_3),$$
 we get Theorem~\ref{mainthm}, i.e.
$${\Biext}^i(K_1,K_2;K_3) \cong {\Ext}^i(K_1 {\buildrel {\scriptscriptstyle \LL} \over \otimes} K_2,K_3) \qquad (i=0,1). $$

\begin{rem} From the exact sequence $0 \rightarrow B_3 \rightarrow K_3 \rightarrow A_3[1] \rightarrow 0$ we get the long exact sequence
\[ 0 \rightarrow \Psi_{\tau_{\leq (1*)}{\rL}'..(K_1,K_2)}^0(B_3) \rightarrow \Psi_{\tau_{\leq (1*)}{\rL}'..(K_1,K_2)}^0 (K_3) \rightarrow \Psi_{\tau_{\leq (1*)}{\rL}'..(K_1,K_2)}^0(A_3[1])\]
\[\rightarrow \Psi_{\tau_{\leq (1*)}{\rL}'..(K_1,K_2)}^1(B_3) \rightarrow \Psi_{\tau_{\leq (1*)}{\rL}'..(K_1,K_2)}^1( K_3)\rightarrow \Psi_{\tau_{\leq (1*)}{\rL}'..(K_1,K_2)}^1(A_3[1]).   \]
The homological interpretation of this long exact sequence is
\[ 0 \rightarrow {\Hom}(T, B_3) \rightarrow {\Hom}(T, K_3) \rightarrow {\Hom}(T, A_3[1])\]
\[\rightarrow {\Ext}^1(T, B_3) \rightarrow {\Ext}^1(T, K_3)\rightarrow {\Ext}^1(T, A_3[1]),   \]
where we set $T={\Tot}\big(\tau_{\leq (1*)}{\rL}'..(K_1,K_2)\big)$, and its geometrical interpretation is
\[ 0 \rightarrow {\Hom}(B_1\otimes B_2, B_3) \rightarrow {\Hom}(K_1 {\buildrel {\scriptscriptstyle \LL} \over \otimes} K_2, K_3) \rightarrow  {\Hom}(A_1 \otimes B_2 + B_1 \otimes A_2, A_3)\]
\[\rightarrow  {\Biext}^1(K_1,K_2; B_3) \rightarrow {\Biext}^1(K_1,K_2; K_3)\rightarrow  {\Hom}(A_1 \otimes A_2, A_3).   \]
\end{rem}
% ------------------------------------------------------------------------

% ---------------------------------------------------------------------------

\end{document}